\documentclass[reqno,11pt]{amsart}
\usepackage{a4wide,color,eucal,enumerate,mathrsfs}
\usepackage[normalem]{ulem}
\usepackage{amsmath,amssymb,epsfig,amsthm,commath} %,bbm
\usepackage[latin1]{inputenc}
\usepackage{psfrag}
\usepackage{hyperref,cleveref,tikz}

\usetikzlibrary{decorations.pathreplacing}

\def\dist{{\mathop\mathrm{dist}}}

\def\ez{\epsilon}
\def\eps{\varepsilon}

\def\bint{{\ifinner\rlap{\bf\kern.35em--}
\int\else\rlap{\bf\kern.45em--}\int\fi}\ignorespaces}

\def\bbint{{\ifinner\rlap{\bf\kern.35em--}
\hspace{0.078cm}\int\else\rlap{\bf\kern.45em--}\int\fi}\ignorespaces}

\def\diam{{\mathop\mathrm{\,diam\,}}}

\newtheorem{thm}{Theorem}[section]
\newtheorem{lem}[thm]{Lemma}%[section]     %@@!!@@!!
\newtheorem{prop}[thm]{Proposition}%[section]    %@@!!@@!!
\numberwithin{equation}{section}

\theoremstyle{remark}
\newtheorem{rem}[thm]{Remark}%[section]    %@@!!@@!!

\def\bint{{\ifinner\rlap{\bf\kern.35em--}
\int\else\rlap{\bf\kern.45em--}\int\fi}\ignorespaces}

\usepackage{graphicx}
\usepackage{import}
\usepackage{xifthen}
\usepackage{pdfpages}
\usepackage{transparent}

\title[Global regularity in the planar Chon\'e-Rochet model]{Global regularity and free boundary geometry \\ in the planar Chon\'e-Rochet model}
\author{Shibing Chen, Alessio Figalli, and Yi Ru-Ya Zhang}
\date{\today}

\address
{School of Mathematical Sciences,
University of Science and Technology of China,
Hefei, Anhui 230026, China}
\email{chenshib@ustc.edu.cn}

 \address{ETH Z\"urich, Department of Mathematics, R\"amistrasse 101, 8092, Z\"urich, Switzerland}
 \email{alessio.figalli@math.ethz.ch}  

\address{Academy of Mathematics and Systems Science, the Chinese Academy of Sciences, Beijing 100190, China}
\email{yzhang@amss.ac.cn}

 \thanks{The first author is supported by NSFC No.12225111, NSFC No.12426202, and NSFC No.12141105. The third author is funded by the National Key R\&D Program of China (Grant No. 2025YFA1018400 \&  No. 2021YFA1003100), NSFC grant No. 12288201 \& No. 12571128, the Chinese Academy of Sciences, and CAS Project for Young Scientists in Basic Research, Grant No. YSBR-031.}

\subjclass[2020]{35R35, 49N60, 35B65, 35Q91}
\keywords{Free boundary problems, convexity constraints, global regularity, obstacle problem.}

\begin{document}

\begin{abstract}
In this paper, we study minimizers of the Chon\'e--Rochet variational problem in dimension two. We first establish global $C^1$ regularity on arbitrary bounded convex domains, and then prove global $C^{1,1}$ regularity on bounded strictly convex domains or, more generally, whenever the zero set of $u$ has positive measure. Next, we construct smooth bounded convex domains with a flat boundary segment for which no prescribed modulus of continuity controls the gradient; this shows that, without additional geometric assumptions, global $C^1$ regularity is optimal. Finally, we prove that the tamed free boundary (that is, the interface between the strictly convex and non-strictly convex regions of the solution) is locally a $C^1$ embedded curve, significantly strengthening previously known regularity results.
\end{abstract}

\maketitle

\section{Introduction}

Let $\Omega \subset \mathbb{R}^2$ be a bounded open convex set, and define
\[
\mathcal{K}(\Omega)
:=
\left\{
u:\overline{\Omega}\to\mathbb{R}
\;\middle|\;
u \text{ is convex and nonnegative}
\right\}.
\]
In this paper, we study the unique minimizer of the variational problem
\begin{equation}\label{monopolist}
    \min_{u \in \mathcal{K}(\Omega)} L[u;\Omega],
    \qquad
    L[u;\Omega]
    :=
    \int_{\Omega}
    \left(
        \tfrac12 |Du-x|^2 + u
    \right)\,dx.
\end{equation}
Since $L[\cdot;\Omega]$ is strictly convex, the minimizer is unique.

Problem \eqref{monopolist} has been studied extensively, both numerically and analytically, over the past two decades; see
\cite{CL2001a, EM2010, MO2014, MRZ20252, MRZ2025, MZ2024}
and the references therein.
Concerning regularity, Rochet and Chon\'e \cite{RC98} and Carlier--Lachand-Robert \cite{CL2001}
proved interior $C^1$ regularity of the minimizer, while Caffarelli and Lions \cite{CL2006}
showed that
\[
u\in C^{1,1}_{\mathrm{loc}}(\Omega).
\]
Although \cite{CL2001} claims global $C^1$ regularity, the argument given there only yields
interior $C^1$ regularity. Our first result fills this gap in the planar case.

\begin{thm}\label{thm:C1}
Let $\Omega\subset\mathbb R^2$ be a bounded convex domain.
Then
\[
u\in C^1(\overline\Omega).
\]
\end{thm}

More recently,  McCann, Rankin and Zhang \cite{MRZ20252} proved that, for polyhedral domains,
the minimizer is $C^{1,1}$ up to the interior of each facet, leaving open the possibility that the
$C^{1,1}$ norm might blow up near the nonsmooth part of the boundary.

Our second main result establishes global $C^{1,1}$ regularity for strictly convex planar domains, and more generally whenever the zero set of $u$ has positive measure.

\begin{thm}\label{C11est}
Let $\Omega \subset \mathbb{R}^2$ be a convex domain, and let $u$ be the minimizer of \eqref{monopolist}. Suppose that there exists $\sigma_0>0$ such that
\[
|\{u=0\}|\ge \sigma_0.
\]
Then there exists a constant $C_0=C_0(\sigma_0,\Omega)$ such that
\[
\|u\|_{C^{1,1}(\Omega)}\le C_0.
\]
In particular, this applies when $\Omega$ is strictly convex.
\end{thm}

We then show that strict convexity is genuinely needed. More precisely, for every modulus of continuity
$\omega$, one can find a bounded smooth convex domain whose boundary contains a line segment and for which
the gradient of the minimizer fails to satisfy the modulus $\omega$.
Thus, without additional geometric assumptions, the global $C^1$ regularity in
Theorem~\ref{thm:C1} is essentially optimal.

\begin{thm}\label{C11fails}
Let $\omega:\mathbb{R}_+\to\mathbb{R}_+$ be a modulus of continuity.
There exists a bounded smooth convex domain $\Omega \subset \mathbb{R}^2$ such that the minimizer $u$ of \eqref{monopolist} satisfies
\[
\sup_{x\neq y\in\Omega}
\frac{|Du(x)-Du(y)|}{\omega(|x-y|)}
=+\infty.
\]
\end{thm}

Beyond the regularity of $Du$, one can study the free boundary separating the strictly convex and
non-strictly convex regions of the minimizer.
For $x\in\Omega$, define the contact set of $u$ with its supporting plane at $x$ by
\[
[x]
:=
\Bigl\{
z\in\overline\Omega:
u(z)=u(x)+Du(x)\cdot(z-x)
\Bigr\}.
\]
It was shown in \cite{MRZ20252} that $\Omega$ decomposes into three disjoint sets,
\[
\Omega=\Omega_0\cup\Omega_1\cup\Omega_2,
\]
where $\Omega_0$ consists of points where $u=0$, $\Omega_1$ consists of points where $u>0$ and $[x]$
is a segment, and $\Omega_2$ consists of points where $u$ is strictly convex.
Moreover, it has been proved that, for every $x\in\Omega_1$, the segment $[x]$ has at least one endpoint on $\partial\Omega$.
Let
\[
\Omega_1'
:=
\Bigl\{
x\in\Omega_1:
[x]\ \text{has exactly one endpoint on }\partial\Omega
\Bigr\}.
\]
As shown in \cite{MRZ20252}\footnote{The argument in \cite{MRZ20252} invokes \cite{CL2001} to justify $C^1$ regularity up to the boundary. As noted above, this fact is not established in \cite{CL2001}, so the validity of their estimates relies instead on our Theorem~\ref{thm:C1}.},
\[
\bigl(Du(x)-x\bigr)\cdot \nu_{\Omega}(x)\ge 0
\]
for every $x\in\partial\Omega$ at which the outer unit normal $\nu_\Omega(x)$ is defined.
The \emph{tamed free boundary} is
\begin{equation}\label{tamefree}
\mathcal{T}
:=
\Bigl\{
x\in \partial\Omega_2\cap\partial\Omega_1\cap\Omega
:\ [x]\cap\partial\Omega=\{x_0\}
\ \text{and}\
\bigl(Du(x_0)-x_0\bigr)\cdot\nu_\Omega(x_0)>0
\Bigr\}.
\end{equation}
For every $x\in\mathcal T$, we denote by $x_0$ the unique point in $[x]\cap\partial\Omega$.
In \cite{MRZ20252}, it was shown that for sufficiently small $\delta>0$, the family of leaves
\[
\{[z]: z\in B_\delta(x_0)\cap\partial\Omega\}
\]
foliates $B_\delta(x_0)$, and these segments extend into $\Omega_2$ without intersecting one another.
Let $v$ denote the convex envelope of $u$ restricted to $B_\delta(x_0)\cap\Omega$.
Then $U:=u-v$ solves, in a neighborhood of $x$,
\[
\Delta U=f\,\chi_{\{U>0\}},
\]
where $f$ is positive and continuous \cite{MRZ20252}.
A main result of \cite{MRZ20252} is that $\mathcal{T}$ has Hausdorff dimension strictly less than $2-\varepsilon$ for some $\varepsilon\in(0,1)$ (see also Section~\ref{sec:related} below for a recent development).
Our next theorem shows that, in fact, $\mathcal{T}$ is locally a $C^1$ embedded curve: for every $x \in \mathcal{T}$ there exists $r_x>0$ such that $\mathcal{T}\cap B_{r_x}(x)$ is a $C^1$ curve.

In view of Theorem~\ref{C11est}, the next result applies whenever $\Omega$ is a strictly convex domain of class $C^{1,1}$. Moreover, by localizing the argument used in its proof, it also applies to the family of square domains $\{(a,a+1)^2\}_{a\ge0}$ considered in \cite{MRZ20252}.

\begin{thm}\label{Fbregu}
Let $\Omega \subset \mathbb{R}^2$ be a convex domain with boundary of class $C^{1,1}$, and assume that $u \in C^{1,1}(\overline\Omega)$.
Then the tamed free boundary $\mathcal{T}$ is locally a $C^1$ embedded curve.
\end{thm}

The proof of Theorem~\ref{Fbregu} combines an iteration scheme, yielding uniqueness of blow-ups at
$\frac12$-density points, with a local rigidity result that rules out cusp-type blow-ups.
The latter statement appears to be of independent interest, and we present it in Theorem~\ref{Rigi} below.
In what follows, we use $(r,\theta)$ to denote polar coordinates.

\begin{thm}\label{Rigi}
Let $u\in C^{1,1}(B_1)$ be a convex function satisfying
\[
u\ge 0 \quad\text{in }B_1\qquad\text{and}
\qquad
u=0 \quad\text{on } S:=\{(x_1,0): -1<x_1<0\}.
\]
Let $a\in(0,\pi/2)$, 
$D:=\{(r,\theta): -\pi+a<\theta<\pi-a\},$
and assume that $\Delta u=1$ in $D\cap B_1$. Then
\[
u(x)=\frac{x_2^2}{2}
\qquad\text{in }D\cap B_1.
\]
In particular, $u$ cannot be strictly convex in $D\cap B_1$.
\end{thm}

\subsection{Motivation of the model}

The minimization problem \eqref{monopolist} is a basic instance of the Chon\'e--Rochet model,
a foundational framework in mechanism design and nonlinear pricing.
A monopolist offers products $y\in\Lambda=[0,\infty)^n$ together with a price schedule
$v:\Lambda\to\mathbb{R}_+$ and seeks to maximize profit under asymmetric information.
In the basic setting considered here, the production cost is
\[
c(y)=\tfrac12 |y|^2,
\qquad y\in\Lambda,
\]
and customer preferences are encoded by
\[
b(x,y)=x\cdot y.
\]
We also assume that the density of customer types is uniform.

Given a price schedule $v$ with $v(0)=0$,
the zero product is available as an outside option at zero cost.
Hence the associated indirect utility
\[
u(x)=\sup_{y\in\Lambda}\bigl(x\cdot y-v(y)\bigr)
\]
is automatically nonnegative, which explains the constraint $u\ge0$ in the definition of
$\mathcal K(\Omega)$.
Moreover, $u$ is convex and differentiable almost everywhere.
Whenever $u$ is differentiable at $x\in\Omega$, the maximizing product is
\[
y(x)=Du(x),
\]
and
\[
v(y(x))=x\cdot Du(x)-u(x).
\]
Therefore the monopolist's profit is
\[
\int_\Omega \bigl(v(y(x))-c(y(x))\bigr)\,dx
=
\int_\Omega
\left(
x\cdot Du(x)-u(x)-\tfrac12 |Du(x)|^2
\right)\,dx.
\]
Since
\[
x\cdot Du(x)-\tfrac12 |Du(x)|^2
=
-\tfrac12 |Du(x)-x|^2+\tfrac12 |x|^2,
\]
maximizing profit is equivalent, up to the additive constant
$\frac12\int_\Omega |x|^2\,dx$, to minimizing \eqref{monopolist}.
In \cite{MRZ20252} the authors study in detail the minimizer of \eqref{monopolist}
for the family of square domains $\{(a,a+1)^2\}_{a\ge0}$, an important example introduced by
Rochet and Chon\'e \cite{RC98}.

\subsection{A related work}\label{sec:related}
Independently and concurrently with this work, the recent preprint \cite{Mc2026} studies the regularity of the free boundary in the same model and shows that it is a continuous curve of Hausdorff dimension one, which is $C^\alpha$ for every $\alpha<1$ away from a discrete set of singular points (see in particular their Theorem~1.1). 
In contrast, our Theorem~\ref{Fbregu} shows that singular points do not occur and that the free boundary is locally a $C^1$ embedded curve.  As observed by the authors of \cite{Mc2026}, our result combines with the argument in the proof of their \cite[Proposition~4.6]{Mc2026} to imply that the tamed free boundary is smooth outside a closed nowhere dense subset.

\subsection{Structure of the paper}

The remainder of the paper is organized as follows.
In Section~\ref{sec2}, we prove Theorem~\ref{thm:C1} and Theorem~\ref{C11est}.
In Section~\ref{sec3}, we prove Theorem~\ref{thm:counterexample}, from which Theorem~\ref{C11fails} follows.
In Section~\ref{sec4}, we prove Theorem~\ref{Fbregu}: we first analyze the $\frac12$-density points via the obstacle problem formulation, and then exclude $0$-density points by establishing the rigidity result of Theorem~\ref{Rigi}.

\section{Global regularity estimates}\label{sec2}

Throughout the manuscript, we denote by $\mathscr H^{1}$ the $1$-dimensional Hausdorff measure. For a set $A$, we denote by $\operatorname{co}(A)$ the convex hull of $A$. A general constant is represented by either $C$ or $c$, and its value may differ across various estimates. All constants it relies on are enclosed in parentheses, indicated as $C(\cdot)$ and $c(\cdot)$. Typically, we use 
$C$ to denote a constant larger than $1$, and $c$ for a constant less than $1$.

In this section, we prove Theorems~\ref{thm:C1} and~\ref{C11est}.
Recall that
\[
u\in C^1(\Omega)
\qquad\text{and}\qquad
u\in C^{1,1}_{\mathrm{loc}}(\Omega)
\]
by \cite[Theorem~2.2]{CL2001} and \cite{CL2006}, respectively.
Moreover, \cite[Theorem~4.1]{MRZ20252} gives a universal $C^{1,1}$ estimate at one-ended leaves.
Thus, in dimension two, the only possible obstruction to a global $C^1$ or $C^{1,1}$ bound comes from points lying on two-ended leaves.

For every $x\in\Omega$, recall that $[x]$ denotes the contact set of
$u$ with its supporting affine function at $x$. We denote by $\Omega_1^0$ the
set of points whose contact set is a segment with two boundary endpoints, and by
\[
\Omega_1':=\Bigl\{x\in\Omega:\ u(x)>0,\ [x]\ \text{is a segment, and }
\#[x]\cap\partial\Omega=1\Bigr\}
\]
the set of points lying on one-ended leaves.

\subsection{Preliminary regularity away from two-ended leaves}

 We begin with a basic interior perturbation estimate, which will serve as the main interior ingredient in Proposition~\ref{prop:goodC11}.
%In what follows, we shall perform integrations by parts that may involve boundary terms containing the gradient of $u$ on $\partial\Omega$. This is not an issue, even though we do not yet know that $u$ is $C^1$ up to the boundary, since for a convex function on a convex domain the gradient admits a unique limit from the interior at $\mathscr H^1$-almost every point of $\partial\Omega$.
% Alternatively, since $u \in C^1(\Omega)$, the integrations by parts below can be
% justified by approximation on smooth inner domains: one performs the argument
% on $\Omega_\varepsilon:=\{x\in\Omega:\dist(x,\partial\Omega)>\varepsilon\}$ and
% then lets $\varepsilon\to 0$.

\begin{lem}\label{lem:interior-perturb}
Fix $x_0\in \Omega\setminus \operatorname{int}\Omega_0$ and $r \in (0, \dist(x_0,\partial\Omega))$.
Set
\[
p_{x_0}(x):=u(x_0)+Du(x_0)\cdot (x-x_0),
\qquad
h:=h_{x_0,r}:=\sup_{B_r(x_0)}(u-p_{x_0}),
\]
and choose $\xi=\xi_{x_0,r}\in \mathbb S^1$ so that
\[
u(x_0+r\xi)-p_{x_0}(x_0+r\xi)=h.
\]
Define the Caffarelli--Lions affine perturbation
\[
p_{x_0,r}^0(x):=p_{x_0}(x)+\frac{h}{2r}\Bigl((x-x_0)\cdot \xi+r\Bigr),
\]
and the associated section and competitor
\[
S_{x_0,r}^0:=\{x\in\Omega:\ u(x)<p_{x_0,r}^0(x)\},
\qquad
u_{x_0,r}^0:=\max\{u,p_{x_0,r}^0\}.
\]
Then
\begin{equation}\label{contain1}
S_{x_0,r}^0\subset \{x\in\Omega:\ -r<(x-x_0)\cdot \xi<r\},
\end{equation}
and, for $r$ sufficiently small,
\begin{equation}\label{CLest}
\int_{S_{x_0,r}^0} |Du_{x_0,r}^0 -Du|^2\,dx\ge c_0\frac{h^2}{r^2}|S_{x_0,r}^0|,
\end{equation}
where $c_0>0$ is universal.

Moreover, if $S_{x_0,r}^0\Subset\Omega$, then
\begin{equation}\label{interior-quadratic}
h_{x_0,r}\le C r^2,
\end{equation}
with $C>0$ universal.
\end{lem}

\begin{proof}
The strip localization \eqref{contain1} is \cite[Eq.~(32)]{MRZ20252}
(see also \cite{CL2006}), and \eqref{CLest} is
\cite[Eq.~(41)]{MRZ20252}.
We now record the variational identity in the form which will be used repeatedly.
Let $p$ be any affine function and set
\[
S:=\{u<p\},
\qquad
u_p:=\max\{u,p\}.
\]
Writing
\[
w:=u_p-u=(p-u)_+,
\]
we have $Dw=Dp-Du$ a.e.\ on $S$ and $w\equiv 0$ on $\partial S\cap\Omega$.
A direct computation gives
\begin{align}
L[u_p;\Omega]-L[u;\Omega]
&=\int_S\Bigl[\frac12\bigl(|Dp|^2-|Du|^2\bigr)-x\cdot(Dp-Du)+w\Bigr]\,dx \nonumber\\
&=-\frac12\int_S|Dw|^2\,dx+\int_S (Dp-x)\cdot Dw\,dx+\int_S w\,dx \nonumber\\
&=-\frac12\int_S|Dw|^2\,dx
+\int_{\partial S\cap\partial\Omega} w\bigl[(Dp-x)\cdot\nu_\Omega\bigr]\,d\mathscr H^1
+3\int_S w\,dx. \label{var-id}
\end{align}
Indeed, since $Dp$ is constant, ${\rm div}(Dp-x)=-2$, and integration by parts gives
\[
\int_S (Dp-x)\cdot Dw
=
\int_{\partial S\cap\partial\Omega} w\bigl[(Dp-x)\cdot\nu_\Omega\bigr]\,d\mathscr H^1
+2\int_S w\,dx.
\]
We apply \eqref{var-id} with $p=p_{x_0,r}^0$. Since $u$ is minimal,
\[
0\le L[u_{x_0,r}^0;\Omega]-L[u;\Omega].
\]
If $S_{x_0,r}^0\Subset\Omega$, the boundary term in \eqref{var-id} vanishes.
Moreover, by \eqref{contain1},
\[
0<u_{x_0,r}^0-u=p_{x_0,r}^0-u\le p_{x_0,r}^0-p_{x_0}\le h
\qquad\text{on }S_{x_0,r}^0,
\]
so that
\[
\int_{S_{x_0,r}^0}(u_{x_0,r}^0-u)\,dx\le h\,|S_{x_0,r}^0|.
\]
Combining this estimate with \eqref{CLest} and \eqref{var-id}, we obtain
\[
0\le
-\frac12c_0\frac{h^2}{r^2}|S_{x_0,r}^0|
+3h\,|S_{x_0,r}^0|,
\]
which is equivalent to \eqref{interior-quadratic} with $C=\frac{6}{c_0}$.
\end{proof}

We can now prove a uniform $C^{1,1}$ regularity away from $\Omega_1^0$.

\begin{prop}\label{prop:goodC11}
Let $U$ be a connected component of $\Omega\setminus \Omega_1^0$.
Then there exists a universal constant $C_0>0$ such that
\[
0\le D^2u\le C_0\,\mathrm{Id}
\qquad\text{a.e.\ on }U.
\]
In particular, $Du$ is $C_0$-Lipschitz on $U$ and admits a unique $C_0$-Lipschitz extension to $\overline U$.

\end{prop}

\begin{proof}
In dimension two, every connected component of
$\Omega\setminus\Omega_1^0$ is convex. Indeed, the leaves of $\Omega_1^0$ are
pairwise disjoint (in $\Omega$) segments; each such segment lies on a line which separates
$\Omega$ into two convex pieces, and a connected component $U$ is obtained by
choosing one side of each leaf and intersecting the corresponding half-planes
with $\Omega$.

We claim that there exists a universal constant $C_0>0$ such that
\[
0\le D^2u\le C_0\,\mathrm{Id}
\qquad\text{a.e.\ on }U.
\]
Indeed, once this is proved, for every $x,y\in U$ the segment $[x,y]$ is
contained in $U$, and the one-dimensional restriction of $u$ to $[x,y]$
yields
\[
|Du(x)-Du(y)|\le C_0|x-y|.
\]
Hence $Du$ is $C_0$-Lipschitz on $U$ and extends uniquely to a Lipschitz map on
$\overline U$.

Now, up to the null set $\partial\Omega_0\cap\Omega$, every point of $U$
belongs to one of the sets
\[
\operatorname{int}\Omega_0,\qquad \Omega_2,\qquad \Omega_1'.
\]
On $\operatorname{int}\Omega_0$ we have $u\equiv 0$, so $D^2u=0$. On $\Omega_2$,
Lemma~\ref{lem:interior-perturb} yields a universal local quadratic estimate,
hence a universal local $C^{1,1}$ bound. On $\Omega_1'$, the one-endpoint
boundary estimate of \cite[Theorem~4.1]{MRZ20252} yields the same universal
local $C^{1,1}$ bound. Since $\Omega_0$ is convex, $\partial\Omega_0$ has zero
Lebesgue measure. Therefore
\[
0\le D^2u\le C_0\,\mathrm{Id}
\qquad\text{a.e.\ on }U,
\]
as claimed.
This proves the proposition.
\end{proof}

\begin{rem}\label{rem:goodC11}
Applying Proposition~\ref{prop:goodC11} in each connected component, we obtain the estimate
\begin{equation}\label{c11good-rmk}
\|D^2u\|_{L^\infty(\Omega_1'\cup\Omega_2)}\le C_0
\end{equation}
for a universal constant $C_0>0$.
\end{rem}

\subsection{Global $C^1$ regularity}

We next prove the geometric lemma for two-ended leaves.

\begin{lem}\label{lower-bound-length}
Let $\Omega\subset\mathbb{R}^2$ be a bounded convex domain, and let $u$ be a minimizer of \eqref{monopolist}. Suppose that there exists a point $x_0\in\Omega$ such that $[x_0]$ is a segment $I\subset\Omega$ whose endpoints $x',x''$ lie on $\partial\Omega$. Let $P'$ and $P''$ be supporting lines to $\Omega$ at $x'$ and $x''$, and let $T$ be the (possibly degenerate) triangular region enclosed by $I$ and the two lines $P',P''$. Then
\[
\{u=0\}\cap\overline{\Omega}\subset T\cap\overline{\Omega}.
\]
\end{lem}

\begin{proof}
Arguing by contradiction, suppose the conclusion fails. Let
$S:=\Omega\cap T,$
so that $S$ is convex and $I\subset \partial S\cap\Omega$. Let
\[
\ell(x):=u(x_0)+Du(x_0)\cdot (x-x_0),
\]
and define
\[
\bar u(x):=
\begin{cases}
\ell(x), & x\in S,\\[1mm]
u(x), & x\in\Omega\setminus S.
\end{cases}
\]

We first record two basic facts.

\smallskip
\noindent\emph{(i) $\bar u$ is convex.}
This follows easily since $\ell$ is a supporting affine function for $u$, and $\ell=u$ on the whole segment $I=[x_0]$.

\smallskip
\noindent\emph{(ii) $\bar u\ge0$ on $\Omega$.}
By the contradiction assumption there exists
\[
y\in \{u=0\}\cap(\overline\Omega\setminus T).
\]
Suppose there exists $x\in S$ with $\ell(x)<0$. Since
$x\in T$ and $y\notin T$, the segment $[x,y]$ has to intersect $I$ at some point 
$z\in I\cap[x,y]$. Since $\ell$ is affine and $\ell(y)\le u(y)=0$, this implies $\ell(z)<0$. On the other hand, $z\in I=[x_0]$, hence $\ell(z)=u(z)\ge0$, a contradiction. This proves that
$\ell\ge0$ on $S$, and consequently $\bar u\ge0$ on all of $\Omega$.

\smallskip
Hence, we have proved that $\bar u\in\mathcal K(\Omega)$. Also, we note that $\bar u\not\equiv u$ (otherwise $u=\ell$ on $S$, so the contact set $[x_0]$ would contain the
two-dimensional set $S$, contradicting the assumption that $[x_0]=I$ is a
segment). Thus, since $u$ is the unique minimizer, we have
\begin{align}
0< L[\bar u;\Omega]-L[u;\Omega]
&= \int_{S}\left[\frac12\bigl(|D\bar u|^2-|Du|^2\bigr)-x\cdot(D\bar u-Du)+(\bar u-u)\right]\,dx \nonumber\\
&\le \int_{S}\frac12\bigl(|D\bar u|^2-|Du|^2\bigr)\,dx - \int_{S}x\cdot(D\bar u-Du)\,dx =: I_1+I_2, \label{bar-u-vari}
\end{align}
where we used $\bar u-u\le0$ on $S$.

We first bound
\begin{align*}
I_1
&=\frac12\int_S \bigl(|D\bar u|^2-|Du|^2\bigr)\,dx\\
&=-\frac12\int_S |D\bar u-Du|^2\,dx+\int_S D\bar u\cdot(D\bar u-Du)\,dx\\
&\le \int_S D\bar u\cdot D(\bar u-u)\,dx.
\end{align*}
Observe now that $\bar u$ is affine on $S$, with
\[
D\bar u\equiv Du(x_0)\qquad\text{on }S.
\]
Integrating by parts and using that $\bar u-u=0$ on
$\partial S\cap\Omega=I$, while $\nu_S=\nu_\Omega$ on
$\partial S\cap\partial\Omega$, we obtain
\[
I_1
\le \int_{\partial S\cap \partial\Omega}(\bar u-u)\,D\bar u\cdot\nu_\Omega\,
d\mathscr H^1.
\]
Similarly,
\begin{align*}
I_2
&=-\int_S x\cdot(D\bar u-Du)\,dx
= -\int_S x\cdot D(\bar u-u)\,dx\\
&=-\int_{\partial S\cap\partial\Omega}(\bar u-u)\,(x\cdot\nu_\Omega)\,
d\mathscr H^1
+2\int_S (\bar u-u)\,dx\\
&\le -\int_{\partial S\cap\partial\Omega}(\bar u-u)\,(x\cdot\nu_\Omega)\,
d\mathscr H^1,
\end{align*}
therefore
\begin{equation}\label{I1+I2}
I_1+I_2\le \int_{\partial S\cap\partial\Omega}(\bar u-u)\,
\bigl[(D\bar u-x)\cdot \nu_\Omega\bigr]\,d\mathscr H^1.
\end{equation}
We now determine the sign of the term
\[
(D\bar u-x)\cdot\nu_\Omega.
\]
Let $\nu'$ and $\nu''$ be outward unit normals to the supporting lines $P'$ and $P''$. By \cite[Proposition~2.3]{MRZ20252}\footnote{The argument in \cite{MRZ20252} is stated under a global $C^1$ regularity assumption. However, their proof applies at all points where $\partial\Omega$ and $u$ are differentiable, which suffices for our purposes.},
\[
(Du(x_0)-x')\cdot\nu'\ge0,
\qquad
(Du(x_0)-x'')\cdot\nu''\ge0.
\]
Since $D\bar u\equiv Du(x_0)$ on $S$, this means exactly that $D\bar u$ belongs
to the closed wedge determined by the two outward supporting half-planes
bounded by $P'$ and $P''$, hence
\[
(D\bar u-x)\cdot \nu_\Omega(x)\ge0
\qquad\text{for }\mathscr H^1\text{-a.e.\ }x\in \partial S\cap \partial\Omega.
\]
Plugging this into \eqref{I1+I2} and recalling that $\bar u-u\le0$ we obtain
\[
I_1+I_2\le0,
\]
which contradicts \eqref{bar-u-vari}. This contradiction proves the lemma.
\end{proof}

Thanks to the previous lemma, we obtain a uniform lower bound on two-ended leaves.

\begin{lem}\label{lem:leaf-lower-bound}
There exists $\delta_0=\delta_0(\Omega)>0$ such that every leaf
$I\subset\Omega_1^0$ satisfies
\[
\diam(I)\ge \delta_0.
\]
\end{lem}

\begin{proof}
By Armstrong's perturbation argument \cite{A1996}, the set $\{u=0\}$ is not a
singleton. Since $\{u=0\}$ is convex, this implies
\[
d_0:=\diam(\{u=0\})>0.
\]
Suppose by contradiction that there exists a sequence of leaves
\[
I_k=[x_k',x_k'']\subset\Omega_1^0,\qquad \text{with}\quad |x_k'-x_k''|\to0.
\]
Let $T_k$ denote the triangular region associated with $I_k$ by
Lemma~\ref{lower-bound-length}. Then
\[
\{u=0\}\cap\overline\Omega\subset T_k\cap\overline\Omega,
\]
hence
\[
d_0\le \diam(T_k\cap\overline\Omega)
\qquad\text{for every }k.
\]
On the other hand, because $\Omega$ is bounded and convex, the diameter of
$T_k\cap\overline\Omega$ tends to zero whenever the chord $I_k$ shrinks to a
point.
This contradicts $d_0\le \diam(T_k\cap\overline\Omega)$ for $k$ large enough.
Thus every two-ended leaf has diameter bounded from below by a universal
constant $\delta_0>0$.
\end{proof}

We now return to the perturbation argument, this time keeping track of the
boundary contribution in the variational identity.
By the global Lipschitz regularity of $u$ \cite{RC98,CL2001}, we may fix
\[
M:=\|Du\|_{L^\infty(\Omega)}<\infty.
\]

\begin{lem}\label{lem:local-perturb}
Fix $x_0\in \Omega\setminus \operatorname{int}\Omega_0$ and
$r \in (0, \dist(x_0,\partial\Omega))$.
Let
\[
h:=h_{x_0,r}:=\sup_{B_r(x_0)}(u-p_{x_0}),
\]
and let $\xi$, $p_{x_0,r}^0$, $S_{x_0,r}^0$, and $u_{x_0,r}^0$ be as in
Lemma~\ref{lem:interior-perturb}. Then
\begin{equation}\label{key-ineq0}
\frac{h}{r^2}\le C(\Omega,M)\left(
1+\frac{\mathscr H^1(\partial S_{x_0,r}^0\cap\partial\Omega)}
{|S_{x_0,r}^0|}
\right).
\end{equation}
Moreover, given a boundary point $x'\in\partial\Omega$, set
\[
d:=|x_0-x'|,
\qquad
\eta:=\frac{x_0-x'}{|x_0-x'|},
\]
and define
\begin{equation}\label{tilted-plane}
p_{x_0,r}(x):=p_{x_0,r}^0(x)+\frac{4h}{d}\Bigl((x-x_0)\cdot\eta+r\,\xi\cdot\eta\Bigr).
\end{equation}
Let
\[
S_{x_0,r}:=\{u<p_{x_0,r}\},
\qquad
u_{x_0,r}:=\max\{u,p_{x_0,r}\}.
\]
If $r\in (0,d/16)$ is sufficiently small, then
\begin{equation}\label{contain2}
S_{x_0,r}\subset \{x\in\Omega:\ -2r<(x-x_0)\cdot \xi<2r\},
\end{equation}
\begin{equation}\label{contain3}
S_{x_0,r}\subset \left\{x\in\Omega:\ (x-x_0)\cdot \eta\ge -\frac34\,d\right\},
\end{equation}
and
\begin{equation}\label{key-ineq}
\frac{h}{r^2}\le C(\Omega,M)\left(
1+\frac{\mathscr H^1(\partial S_{x_0,r}\cap\partial\Omega)}
{|S_{x_0,r}|}
\right).
\end{equation}
\end{lem}

\begin{proof}
We first derive the estimate for the untilted section. By
Lemma~\ref{lem:interior-perturb}, the strip inclusion \eqref{contain1}, the
lower bound \eqref{CLest}, and the variational identity \eqref{var-id} all hold.

For \(x\in B_r(x_0)\) we have
\[
u(x)-p_{x_0}(x)
\le |u(x)-u(x_0)|+|Du(x_0)|\,|x-x_0|
\le Mr+Mr=2Mr.
\]
Hence
\[
h\le 2Mr,
\qquad\text{and therefore}\qquad
\frac{h}{r}\le 2M.
\]
It follows that
\[
|Dp_{x_0,r}^0|
= \left|Du(x_0)+\frac{h}{2r}\xi\right|
\le M+\frac{h}{2r}\le 2M.
\]
Applying \eqref{var-id} with \(p=p_{x_0,r}^0\) and using minimality of \(u\), we obtain
\begin{align*}
0
&\le -\frac12\int_{S_{x_0,r}^0}|Du_{x_0,r}^0-Du|^2\,dx \\
&\quad +\int_{\partial S_{x_0,r}^0\cap\partial\Omega}
(u_{x_0,r}^0-u)\bigl[(Dp_{x_0,r}^0-x)\cdot\nu_\Omega\bigr]\,d\mathscr H^1
+3\int_{S_{x_0,r}^0}(u_{x_0,r}^0-u)\,dx.
\end{align*}
Now \eqref{contain1} gives
\[
0<u_{x_0,r}^0-u=p_{x_0,r}^0-u\le p_{x_0,r}^0-p_{x_0}\le h
\qquad\text{on }S_{x_0,r}^0,
\]
hence
\[
\int_{S_{x_0,r}^0}(u_{x_0,r}^0-u)\,dx\le h\,|S_{x_0,r}^0|.
\]
Also
\[
|(Dp_{x_0,r}^0-x)\cdot\nu_\Omega|
\le |Dp_{x_0,r}^0|+|x|
\le C(\Omega,M)
\qquad\text{on }\partial\Omega.
\]
Combining these bounds with \eqref{CLest} yields
\[
0\le
-\frac12 c_0\frac{h^2}{r^2}|S_{x_0,r}^0|
+C(\Omega,M)\,h\,\mathscr H^1(\partial S_{x_0,r}^0\cap\partial\Omega)
+3h\,|S_{x_0,r}^0|,
\]
and dividing by \(h\,|S_{x_0,r}^0|\) gives \eqref{key-ineq0}.

We turn to the tilted construction. The normalization in \eqref{tilted-plane}
is chosen so that
\[
p_{x_0,r}(x_0-r\xi)=p_{x_0,r}^0(x_0-r\xi).
\]
This is exactly the normalization used in \cite[Proof of Theorem~4.1]{MRZ20252}
to preserve the strip localization under tilting. For \(r\le d/16\) one has
\[
p_{x_0,r}(x_0)-u(x_0)
=
\frac h2+\frac{4h}{d}\,r\,\xi\cdot\eta
\ge \frac h2-\frac{4hr}{d}
\ge \frac h4,
\]
hence \(x_0\in S_{x_0,r}\). Moreover, by
\cite[Proof of Theorem~4.1]{MRZ20252}, for \(r\) sufficiently small
\eqref{contain2}, \eqref{contain3}, and the analogue of \eqref{CLest} hold:
\[
\int_{S_{x_0,r}} |Du_{x_0,r}-Du|^2\,dx
\ge c_0\frac{h^2}{r^2}|S_{x_0,r}|
\]
for some universal \(c_0>0\).
Using again \(h\le 2Mr\) and \(r<d\), we get
\[
|Dp_{x_0,r}|
=
\left|Du(x_0)+\frac{h}{2r}\xi+\frac{4h}{d}\eta\right|
\le M+\frac{h}{2r}+\frac{4h}{d}
\le C\,M.
\]
Also, as in \cite[Proof of Theorem~4.1]{MRZ20252}, the localizations
\eqref{contain2}--\eqref{contain3} imply
\[
0<u_{x_0,r}-u\le C h
\qquad\text{on }S_{x_0,r},
\]
hence
\[
\int_{S_{x_0,r}}(u_{x_0,r}-u)\,dx\le C h |S_{x_0,r}|.
\]
Applying \eqref{var-id} with \(p=p_{x_0,r}\), and arguing exactly as above, we
obtain
\[
0\le
-\frac12 c_0\frac{h^2}{r^2}|S_{x_0,r}|
+C(\Omega,M)\,h\,\mathscr H^1(\partial S_{x_0,r}\cap\partial\Omega)
+C h\,|S_{x_0,r}|,
\]
which gives \eqref{key-ineq}.
\end{proof}

The next step is to prove a uniform $C^{1,1}$ regularity in a strip between two leaves.

\begin{lem}\label{lem:strip-C11}
Let \(I_1,I_2\subset \Omega_1^0\) be two distinct leaves, and let
\(V\subset \Omega\setminus (I_1\cup I_2)\) be the connected component between
\(I_1\) and \(I_2\). Then
\[
\|D^2u\|_{L^\infty(\overline V\cap\Omega)}\le C(\Omega,I_1,I_2).
\]
In particular, two distinct leaves of \(\Omega_1^0\) cannot meet at the same
boundary point.
\end{lem}

\begin{proof}
By Proposition~\ref{prop:goodC11}, it is enough to prove a uniform quadratic
estimate at points
\[
x_0\in V\cap \Omega_1^0.
\]
Fix such a point, and write
\[
[x_0]=[x_0',x_0''],
\qquad
\tau_0:=\frac{x_0''-x_0'}{|x_0''-x_0'|},
\]
where
\begin{equation}\label{x' small}
|x_0-x_0'|\le |x_0-x_0''|.
\end{equation}
Let \(T_1,T_2,T_0\) be the triangular regions associated with
\(I_1,I_2,[x_0]\), as in the statement of Lemma~\ref{lower-bound-length}.
Since \([x_0]\) lies between \(I_1\) and \(I_2\), these triangles are nested.
Consequently, there exists a constant
$\delta=\delta(\Omega,I_1,I_2)>0$
such that
\begin{equation}\label{eq:strip-length}
|x_0'-x_0''|\ge \delta,
\end{equation}
and one can choose supporting lines \(P'\), \(P''\) to \(\Omega\) at \(x_0'\),
\(x_0''\), with outward unit normals \(\nu'\), \(\nu''\), such that
\begin{equation}\label{eq:strip-angle}
|\tau_0\cdot \nu'|\ge \delta,
\qquad
|\tau_0\cdot \nu''|\ge \delta.
\end{equation}
Now fix
\[
0<r<\dist(x_0,\partial\Omega),
\qquad
h:=\sup_{B_r(x_0)}(u-p_{x_0}),
\]
and apply the tilted construction of Lemma~\ref{lem:local-perturb}, using
\(x_0'\) as the endpoint with respect to which we tilt. Denoting the resulting
section by
$S:=S_{x_0,r},$
estimate \eqref{key-ineq} gives
\begin{equation}\label{eq:strip-key}
\frac{h}{r^2}\le C(\Omega,M)\left(
1+\frac{\mathscr H^1(\partial S\cap\partial\Omega)}{|S|}
\right).
\end{equation}
Therefore it remains to prove
\begin{equation}\label{eq:strip-boundarysize}
\mathscr H^1(\partial S\cap\partial\Omega)\le C(\Omega,I_1,I_2)\,|S|.
\end{equation}

Let
$A:=\partial S\cap\partial\Omega.$
Then the tilt with respect to the endpoint \(x_0'\), together with
\eqref{x' small}, \eqref{eq:strip-length}, and \eqref{eq:strip-angle}, implies
that for all sufficiently small \(r\):

\smallskip
\noindent
(a) the boundary trace \(A\) is contained in a single boundary chart near the
endpoint \(x_0''\);

\smallskip
\noindent
(b) one has \(x_0\in S\).

Indeed, property (b) follows from Lemma~\ref{lem:local-perturb}, since we are
tilting with respect to the endpoint \(x_0'\) chosen so that
\eqref{x' small} holds.

Since \(P''\) is a supporting line at \(x_0''\), \eqref{x' small},
\eqref{eq:strip-length}, and \eqref{eq:strip-angle} yield
\[
\dist(x_0,P'')
=
|(x_0-x_0'')\cdot \nu''|
=
|x_0-x_0''|\,|\tau_0\cdot \nu''|
\ge \frac12 |x_0'-x_0''|\,|\tau_0\cdot \nu''|
\ge \frac{\delta^2}{2}.
\]
If \(A=\emptyset\), then \eqref{eq:strip-boundarysize} is trivial. Assume
\(A\neq\emptyset\). Since \(\overline S\) is convex and contains both \(A\) and
\(x_0\), we have
\[
\operatorname{co}(A\cup\{x_0\})\subset \overline S.
\]
Flattening \(\partial\Omega\) in the boundary chart near \(x_0''\), the length
of \(A\) is comparable to the length of its orthogonal projection onto \(P''\),
with constants depending only on \(\Omega\). Therefore
\[
|S|
\ge \bigl|\operatorname{co}(A\cup\{x_0\})\bigr|
\ge c(\Omega)\,\dist(x_0,P'')\,\mathscr H^1(A)
\ge c(\Omega,\delta)\,\mathscr H^1(A),
\]
which is \eqref{eq:strip-boundarysize}.

Plugging \eqref{eq:strip-boundarysize} into \eqref{eq:strip-key}, we obtain
\[
h\le C(\Omega,I_1,I_2)\,r^2.
\]
Since \(x_0\in V\cap \Omega_1^0\) was arbitrary, this proves a uniform local
\(C^{1,1}\) estimate at every point of \(V\cap \Omega_1^0\). Together with
Proposition~\ref{prop:goodC11} on \(\overline V\setminus \Omega_1^0\), we
conclude that
\[
\|D^2u\|_{L^\infty(\overline V\cap\Omega)}\le C(\Omega,I_1,I_2).
\]

For the final statement, suppose that \(I_1\) and \(I_2\) meet at the same
boundary point \(z\in\partial\Omega\). By the first part of the lemma,
\(Du\) extends continuously to \(z\) from inside \(V\); call the limit \(q\).
Also, given a leaf \(I\), let \(p(I)\) denote the (constant) gradient of \(u\)
on \(I\). Since \(Du\) is constant on each leaf, we must have
\[
p(I_1)=p(I_2)=q.
\]
Let
\[
\ell(x):=u(z)+q\cdot (x-z).
\]
Because \(q\in \partial u(z)\), we have \(u\ge \ell\) in \(\Omega\), and
\(u=\ell\) on \(I_1\cup I_2\). Hence the nonnegative convex function
\(u-\ell\) vanishes on the convex hull of \(I_1\cup I_2\), and therefore
\[
u=\ell \qquad\text{on } \operatorname{co}(I_1\cup I_2).
\]
Since \(I_1\) and \(I_2\) are distinct and share the point \(z\), this convex
hull contains a two-dimensional subset of \(\Omega\), contradicting the fact
that points of \(\Omega_1^0\) have one-dimensional contact set. This
contradiction shows that two distinct leaves cannot meet at the same boundary
point.
\end{proof}

Until now, we have only considered leaves contained inside $\Omega$. However, a limit of such leaves may converge to a segment on the boundary of $\Omega$ if the domain is not strictly convex. For this reason, we denote
\[
\mathscr L_1^0:=\{[x]:x\in \Omega_1^0\}
\]
the family of two-ended leaves, viewed as a subset of the space of nonempty
compact subsets of \(\overline\Omega\) endowed with the Hausdorff distance, and we define \(\overline{\mathscr L_1^0}^{\,H}\) to be
its Hausdorff closure. By Lemma~\ref{lem:leaf-lower-bound}, every element of
\(\overline{\mathscr L_1^0}^{\,H}\) is a segment of positive length with
endpoints on \(\partial\Omega\). Moreover, if \(I_k\in \mathscr L_1^0\) and
\(I_k\to I\) in Hausdorff distance, then passing to the limit in the supporting
affine functions shows that \(u\) agrees on \(I\) with a supporting affine
function. We call the elements of \(\overline{\mathscr L_1^0}^{\,H}\)
\emph{extended two-ended leaves}.

We now show that along sequences inside $\Omega_1^0$, the gradient of $u$ has a unique limit on the boundary.

\begin{lem}\label{lem:leaf-slope-limit}
Let $x_0\in \partial\Omega\cap \overline{\Omega_1^0}$.
Then there exists a unique vector
$q(x_0)\in \mathbb R^2$
such that, for every sequence
$x_k\in \Omega_1^0$ with
$x_k\to x_0,$
one has
\[
Du(x_k)\to q(x_0).
\]
\end{lem}

\begin{proof}
Fix a sequence
\[
x_k\in \Omega_1^0,
\qquad
x_k\to x_0,
\]
and write
\[
I_k:=[x_k]\in \mathscr L_1^0,
\qquad
p_k:=Du(x_k).
\]
Since the leaves have uniformly positive length by
Lemma~\ref{lem:leaf-lower-bound}, after passing to a subsequence we may assume
that
\[
I_k\to I_*
\qquad\text{in the Hausdorff topology}
\]
for some extended two-ended leaf \(I_*\in \overline{\mathscr L_1^0}^{\,H}\).
Moreover, since \(x_k\in I_k\) and \(x_k\to x_0\), we have
\[
x_0\in I_*.
\]
We distinguish two cases.

\smallskip
\noindent\emph{Case 1: \(I_*\) is a genuine leaf.}
Since \(x_0\in \partial\Omega\) and \(\operatorname{relint}(I_*)\subset \Omega\),
the point \(x_0\) must be an endpoint of \(I_*\), so \(I_*=[x_0,b_*]\) for some \(b_*\in \partial \Omega\).
Let
\[
p_*:=Du\qquad\text{on }I_*,
\]
and choose two disjoint boundary arcs
\[
A\ni x_0,
\qquad
B\ni b_*,
\]
so small that every segment sufficiently close to \(I_*\) in the Hausdorff
topology and meeting a sufficiently small neighborhood of \(x_0\) has exactly
one endpoint on \(A\) and one endpoint on \(B\). Fix \(r_0>0\) and
\(\varepsilon_0>0\) so small that every leaf \(J\in \mathscr L_1^0\) satisfying
\[
d_H(J,I_*)\le \varepsilon_0,
\qquad
J\cap \overline B_{r_0}(x_0)\neq\varnothing
\]
has this property, and consider the family
\[
\mathcal F:=
\Bigl\{J\in \mathscr L_1^0:\ d_H(J,I_*)\le \varepsilon_0,\ 
J\cap \overline B_{r_0}(x_0)\neq\varnothing\Bigr\}.
\]
Since all leaves sufficiently close to \(I_*\) must intersect \(\Omega\), the family \(\mathcal F\) is compact in the Hausdorff topology.

For \(J=[a(J),b(J)]\in \mathcal F\), set
\[
\tau(J):=\frac{b(J)-a(J)}{|b(J)-a(J)|},
\qquad
\alpha(J):=\sup_{\nu\in N_\Omega(b(J))} |\tau(J)\cdot \nu|,
\]
where \(N_\Omega(b(J))\) denotes the set of all unit outward normals to \(\Omega\) at \(b(J)\).
By Lemma~\ref{lem:leaf-lower-bound},
\[
\mathscr H^1(J)\ge \delta_0(\Omega)>0
\qquad\text{for every }J\in \mathcal F.
\]
Also, because leaves in \(\mathcal F\) cannot be tangent to \(\partial\Omega\) at
their endpoint \(b(J)\), compactness gives
\[
\delta_*:=\inf_{J\in \mathcal F}\alpha(J)>0.
\]
Now let \(J_1,J_2\in \mathcal F\), and let \(V\) be the connected component of
\(\Omega\setminus (J_1\cup J_2)\) lying between them. The proof of
Lemma~\ref{lem:strip-C11} depends only on a lower bound on the leaf length, a lower bound on the angles, and the local geometry of \(\Omega\). Therefore there exists \(C_*>0\), depending only on \((x_0,I_*,\Omega)\), such that
\[
\|u\|_{C^{1,1}(\overline V\cap B_{r_0}(x_0))}\le C_*.
\]
Since \(I_k\to I_*\), we have \(I_k\in\mathcal F\) for \(k\) large. Choose
\(0<\rho<r_0\) so small that every \(J\in \mathcal F\) intersects
\(\partial B_\rho(x_0)\) in exactly one point, denoted by \(z(J)\).
Then \(z(I_k)\to z(I_*)\). Also, if \(V_k\) denotes the strip between \(I_k\) and \(I_*\),
\[
\|D^2u\|_{L^\infty(\overline V_k\cap B_{r_0}(x_0))}\le C_*.
\]
Thus, since \(Du\) is constant on each leaf,
\[
|p_k-p_*|
=
|Du(z(I_k))-Du(z(I_*))|
\le C_*\,|z(I_k)-z(I_*)|
\to 0,
\]
which proves that \(p_k\to p_*\).

\smallskip
\noindent\emph{Case 2: \(I_*\subset \partial\Omega\).}
In this case, \(I_*\) is a nondegenerate boundary segment containing \(x_0\). Let
\(\tau\) be a unit vector parallel to \(I_*\), and let \(K\) be the tangent cone
of \(\Omega\) at \(x_0\), i.e., the set of all limits of directions of points in \(\Omega\) approaching \(x_0\). Since \(I_*\subset \partial\Omega\), the vector \(\tau\)
spans one edge of \(K\). Choose two linearly independent unit vectors
\[
e_1,e_2\in \operatorname{int} K
\]
such that neither \(e_i\) is parallel to \(\tau\). Fix a point \(x_*\in \operatorname{relint}(I_*)\) and, for \(\varepsilon>0\) small, consider the segments
\[
\Gamma_i:=\{x_*+t e_i:\ 0\le t\le \varepsilon\},
\qquad i=1,2,
\]
which are contained in \(\overline\Omega\).
Define
\[
\psi_i(t):=u(x_*+t e_i),
\qquad 0\le t\le \varepsilon.
\]
Since \(u\) is convex, each \(\psi_i\) is convex on \([0,\varepsilon]\), hence
the right derivative \(\psi_i'(0^+)\) exists.
Because \(I_k\to I_*\) in the Hausdorff topology, the direction of \(I_k\)
converges to \(\tau\). Since each \(\Gamma_i\) is transversal to \(\tau\), for
\(k\) sufficiently large the leaf \(I_k\) intersects \(\Gamma_i\) in exactly one
point, say
\[
z_i^k=x_*+t_i^k e_i,
\qquad 0<t_i^k<\varepsilon,
\qquad i=1,2,
\]
with \(t_i^k\to 0\) as \(k\to\infty\).
Since \(z_i^k\in I_k\subset \Omega\) and \(Du\equiv p_k\) on \(I_k\), we have
\[
p_k\cdot e_i
=
Du(z_i^k)\cdot e_i
=
\psi_i'(t_i^k),
\qquad i=1,2,
\]
and by the monotonicity of \(\psi_i'\), it follows that
\[
\psi_i'(t_i^k)\to \psi_i'(0^+),
\qquad i=1,2.
\]
Hence
\[
p_k\cdot e_i\to \psi_i'(0^+),
\qquad i=1,2.
\]
Since \(e_1,e_2\) are linearly independent, there exists a unique vector \(q_{\rm tan}(x_0)\in \mathbb R^2\)
such that
\[
q_{\rm tan}(x_0)\cdot e_i=\psi_i'(0^+),
\qquad i=1,2.
\]
This proves that
\[
p_k\to q_{\rm tan}(x_0).
\]

We now prove that the limit is independent of the chosen subsequence.
If two subsequences fall under Case~1, with limiting genuine leaves
\[
I_*'=[x_0,b'],
\qquad
I_*''=[x_0,b''],
\]
then \(I_*'=I_*''\) by the last statement of Lemma~\ref{lem:strip-C11}, since
two distinct leaves cannot meet at the same boundary point. Hence the limits coincide.

If two subsequences fall under Case~2, then both limits equal \(q_{\rm tan}(x_0)\), because this vector is determined only by the fixed rays \(\Gamma_1,\Gamma_2\) and the traces of \(u\) on them.

Assume now that one subsequence falls under Case~1, with limiting genuine leaf
$I_*=[x_0,b_*],$
and another subsequence falls under Case~2, with limiting boundary segment
$J_*\subset\partial\Omega$
containing \(x_0\). We claim that this is impossible.
Indeed, let \(J_k\) be the leaves from the Case~2 subsequence. Since leaves do not cross,
they are totally ordered in a neighborhood of \(x_0\). 
Moreover, since \(I_*\) is a genuine leaf issuing from \(x_0\), again by the noncrossing
property each \(J_k\) must lie on the same side of \(I_*\). Hence, it follows that all the leaves \(\{J_k\}_{k \geq k_0}\) are contained in the region $V_{k_0}$ between \(I_*\) and \(J_{k_0}\).
In particular, choosing $k_0$ large enough so that $J_*\setminus \partial V_{k_0}\neq \emptyset$, the leaves \(\{J_k\}_{k \geq k_0}\) remain at a positive distance from a nontrivial portion of \(J_*\),
and therefore cannot converge to \(J_*\) in the Hausdorff topology.
This contradiction shows that the mixed case cannot occur.

In conclusion, every convergent subsequence of \(\{p_k\}\) has the same limit, and since the initial sequence \(\{x_k\}\) was arbitrary, this proves the lemma.
\end{proof}

We can now prove the global $C^1$ regularity in dimension two.

\begin{proof}[Proof of Theorem~\ref{thm:C1}]
Fix \(x_0\in\partial\Omega\).
If \(x_0\notin \overline{\Omega_1^0}\), then there exists \(r>0\) such that
\[
B_r(x_0)\cap\Omega_1^0=\emptyset.
\]
Proposition~\ref{prop:goodC11} then implies that \(Du\) admits a Lipschitz extension to \(B_r(x_0)\cap\overline\Omega\), and the conclusion follows.
Hence, we can assume that
\[
x_0\in \partial\Omega\cap \overline{\Omega_1^0}.
\]
By Lemma~\ref{lem:leaf-slope-limit}, there exists a unique vector
$q(x_0)\in\mathbb R^2$
such that
\begin{equation}\label{eq:Du Omega10}
Du(z_k)\to q(x_0)
\qquad\text{for every sequence }z_k\in \Omega_1^0,\ z_k\to x_0.
\end{equation}
Now let \(y_k\to x_0\) with \(y_k\in \Omega\setminus\Omega_1^0\), denote by \(U_k\)
the connected component of \(\Omega\setminus \Omega_1^0\) containing \(y_k\), and
choose
$z_k\in \partial U_k\cap\Omega_1^0$
such that
\[
|y_k-z_k|\leq 2\dist(y_k,\Omega_1^0).
\]
Since \(x_0\in \overline{\Omega_1^0}\), we have
\[
z_k\to x_0,
\qquad
|y_k-z_k|\to 0.
\]
By Proposition~\ref{prop:goodC11}, \(Du\) is \(C_0\)-Lipschitz on \(U_k\) and extends with the same Lipschitz constant to \(\overline U_k\). Hence, thanks to \eqref{eq:Du Omega10},
\[
|Du(y_k)-q(x_0)|
\le |Du(y_k)-Du(z_k)|+|Du(z_k)-q(x_0)|
\le C_0|y_k-z_k|+|Du(z_k)-q(x_0)|
\to 0.
\]

In conclusion, we have shown that
\[
Du(x)\to q(x_0)\qquad\text{as }x\to x_0,\ x\in\Omega.
\]
Since \(x_0\in\partial\Omega\) was arbitrary, \(Du\) extends continuously to \(\overline\Omega\), and therefore
$u\in C^1(\overline\Omega).$
\end{proof}

\subsection{Global $C^{1,1}$ regularity in strictly convex domains}

\begin{proof}[Proof of Theorem~\ref{C11est}]
By Proposition~\ref{prop:goodC11}, it remains to prove the estimate at points
$x_0\in \Omega_1^0$.
Fix such an $x_0$, and write
\[
[x_0]=[x',x''].
\]
Let
\[
\tau:=\frac{x''-x'}{|x''-x'|}.
\]
By Lemma~\ref{lem:leaf-lower-bound}, there exists $\delta_0=\delta_0(\Omega)>0$
such that
\begin{equation}\label{leaf-lower-bound-C11}
|x'-x''|\ge \delta_0
\qquad\text{for every }x_0\in\Omega_1^0.
\end{equation}

We claim that there exists
$\delta_1=\delta_1(\Omega)>0$ such that for every $x_0\in\Omega_1^0$, both of the endpoints $x',\,x''$ of the leaf,  and the corresponding supporting lines $P'$ and $P''$ to
$\Omega$ at $x'$ and $x''$ with outward unit normals $\nu'$ and $\nu''$ satisfying
\begin{equation}\label{good-endpoint}
|\tau\cdot \nu'|\ge \delta_1 \quad \text{ and }
\quad |\tau\cdot \nu''|\ge \delta_1,
\end{equation}
respectively. 
Indeed, if \eqref{good-endpoint} fails, for instance, at $x''$,  we could find a sequence of leaves with
$|x'-x''|\ge\delta_0$ and supporting normals at one endpoint such that
$|\tau\cdot\nu''|\to0$. 
Passing to a subsequence, the endpoints converge to two distinct boundary points, and the corresponding chord directions converge. The limiting chord would then be tangent to $\partial\Omega$ at one endpoint, which is impossible by Lemma~\ref{lower-bound-length} together with the assumption that $|\{u=0\}|\ge \sigma_0>0$. Hence \eqref{good-endpoint} holds.

By symmetry, we may label $x'$ and $x''$ so that
\begin{equation}\label{retained-point}
 |x_0-x'|\le |x_0-x''|.   
\end{equation} 
We tilt exactly as in Lemma~\ref{lem:local-perturb} (cf.\ \cite[(33)]{MRZ20252})
with respect to the endpoint $x'$, so as to shift the boundary  interaction to $x''$.
% For uniformity, we simply apply the tilted construction
% with respect to $x'$.
Thus, for all sufficiently small $r$, we obtain a section
$S:=S_{x_0,r}$
for which \eqref{key-ineq} holds. It is therefore enough to prove
\begin{equation}\label{boundarysize}
\mathscr H^1(\partial S\cap\partial\Omega)\le C(\Omega)\,|S|.
\end{equation}
The tilted section $S$ has the
following two properties for all sufficiently small $r$:

\smallskip
\noindent
(a) its boundary trace
$A:=\partial S\cap\partial\Omega$
is near the  endpoint $x''$;

\smallskip
\noindent
(b) one has $x_0\in S$.

Since $P''$ is a supporting line at $x''$, \eqref{retained-point} and
\eqref{good-endpoint} imply
\[
\dist(x_0,P'')
=
|(x_0-x'')\cdot \nu''| = |x_0-x''||\tau \cdot \nu''|
\ge c(\Omega).
\]
If $A=\emptyset$, then \eqref{boundarysize} is trivial. Assume $A\neq\emptyset$.
Because $\overline S$ is convex and contains both $A$ and $x_0$, we have
\[
\operatorname{co}(A\cup\{x_0\})\subset \overline S.
\]
Flattening $\partial\Omega$ in the boundary chart around $x''$, the length of
$A$ is comparable to the length of its orthogonal projection onto $P''$, with
constants depending only on $\Omega$. Therefore
\[
|S|
\ge \bigl|\operatorname{co}(A\cup\{x_0\})\bigr|
\ge c(\Omega)\,\dist(x_0,P'')\,\mathscr H^1(A)
\ge c(\Omega)\,\mathscr H^1(A),
\]
which is exactly \eqref{boundarysize}.

Plugging \eqref{boundarysize} into \eqref{key-ineq} gives $h_{x_0,r}\le C(\Omega)\,r^2$, therefore 
\[
u(x)-p_{x_0}(x)\le C(\Omega)\,|x-x_0|^2
\qquad\text{for }x\text{ near }x_0.
\]
Since $u$ is convex, this is equivalent to a local $L^\infty$ bound on $D^2u$ at $x_0$.
Together with Proposition~\ref{prop:goodC11}, this yields $\|D^2u\|_{L^\infty(\Omega)}\le C(\Omega)$.
Since the zero set 
$\{u=0\}$ has positive measure, $Du=0$ there. Hence, because $\Omega$ is convex, the Hessian bound controls the full $C^{1,1}$ norm, and we obtain
\[
\|u\|_{C^{1,1}(\Omega)}\le C_0(\Omega),
\]
as claimed.

When $\Omega$ is strictly convex, the existence of $\sigma_0>0$ follows directly from the  Armstrong's perturbation argument \cite{A1996}.
\end{proof}

\section{A counterexample to the global $C^{1,1}$ regularity}\label{sec3}

We present an example showing that, in general, one cannot expect a uniform
$C^{1,1}$ bound for minimizers on smooth convex domains whose boundary contains
flat portions.

%------------------------------------------------------------
\subsection{Construction of the domain}

Fix $\epsilon\in(0,1)$ and $a>0$, let $C_0>1$ denote the universal constant in
\eqref{c11good-rmk} from Remark~\ref{rem:goodC11}, and set
\[
\eta_\epsilon:=\frac{1}{1000\,C_0}\,\epsilon.
\]
Choose a concave function $g\in C^\infty((0,1))\cap C^0([0,1])$ such that
\begin{equation}\label{g1}
g(0)=0,\qquad g(x_1)=\frac14+x_1\ \ \text{for }x_1\in\Bigl(\frac12,\,1-\eta_\epsilon\Bigr),
\qquad g'(x_1)\to +\infty \ \text{as }x_1\to 0^+,
\end{equation}
and
\begin{equation}\label{g2}
g(1)=\frac54-\eta_\epsilon,
\qquad g'(x_1)\to -\infty \ \text{as }x_1\to 1^-.
\end{equation}
Moreover, we choose $g$ so that the embedded curve obtained by adjoining the vertical rays
$\{(0,x_2): x_2\le 0\}$ and $\{(1,x_2): x_2\le g(1)\}$ to the graph of $g$ over $(0,1)$ is $C^\infty$.

Define
\[
g_{\epsilon,a}(x_1):=\epsilon+g(x_1-a)\qquad\text{for }x_1\in[a,a+1],
\]
and let
\[
\Omega_{\epsilon,a}:=\{(x_1,x_2): a\le x_1\le a+1,\ |x_2|\le g_{\epsilon,a}(x_1)\}
\]
be the smooth bounded domain whose boundary consists of
\[
\{(a,x_2): -\epsilon\le x_2\le \epsilon\},\qquad
\{(a+1,x_2): -\epsilon-\tfrac54+\eta_\epsilon\le x_2\le \epsilon+\tfrac54-\eta_\epsilon\},
\]
and the two graphs
\[
\{(x_1,x_2): x_2=\pm g_{\epsilon,a}(x_1),\ a<x_1<a+1\}.
\]
Since $g_{\epsilon,a}$ is concave and positive, $\Omega_{\epsilon,a}$ is convex. Also, by construction,
\begin{equation}\label{diam-bound}
\diam(\Omega_{\epsilon,a})\le 5,\qquad |\Omega_{\epsilon,a}|\ge \frac34.
\end{equation}

\begin{figure}[t]
\centering
\begin{tikzpicture}[x=5.2cm,y=2.4cm,>=stealth,scale=0.7]

%--- parameters for the schematic (for drawing only)
\def\eps{0.18}   % corresponds to \epsilon
\def\eta{0.04}   % corresponds to \eta_\epsilon (small)
\pgfmathsetmacro{\H}{1.25-\eta+\eps} % = \epsilon + 5/4 - \eta_\epsilon

%--- rounding amount near x_1=a+1 (drawing only)
\def\rho{0.10}

%--- domain (schematic boundary with vertical tangents at x=0 and rounded corners at x=1)
\path[fill=gray!12,draw=black,thick]
  (0,-\eps) -- (0,\eps)
  .. controls (0,{0.78}) and ({1-2*\rho},{\H}) .. ({1-\rho},{\H})
  arc[start angle=90,end angle=0,radius=\rho]
  -- (1,{-\H+\rho})
  arc[start angle=0,end angle=-90,radius=\rho]
  .. controls ({1-2*\rho},{-\H}) and (0,{-0.78}) .. (0,-\eps)
  -- cycle;

%--- axes
\draw[->] (-1.6,0) -- (1.18,0) node[right] {$x_1$};
\draw[->] (-1.5,-1.35) -- (-1.5,1.35) node[above] {$x_2$};

%--- mark a and a+1
\node[below=4pt] at (0.05,0) {$a$};
\node[below=4pt] at (1.15,0) {$a+1$};

%--- brace
\draw[decorate,decoration={brace,amplitude=4pt}]
  (-0.02,-\eps) -- (-0.02,\eps)
  node[midway,left=6pt] {$2\epsilon$};

%--- labels
\node at (0.55,0.35) {$\Omega_{\epsilon,a}$};
\node[above] at (0.32,\H-0.12) {$x_2=\epsilon+g(x_1-a)$};
\node[below] at (0.32,-\H+0.12) {$x_2=-\epsilon-g(x_1-a)$};

%--- emphasize the flat sides
\draw[black,very thick] (0,-\eps) -- (0,\eps);
\draw[black,very thick] (1,{-\H+\rho}) -- (1,{\H-\rho});

\end{tikzpicture}
\caption{Schematic of the convex domain
$\Omega_{\epsilon,a}=\{(x_1,x_2):a\le x_1\le a+1,\ |x_2|\le \epsilon+g(x_1-a)\}$.
The curved parts represent the graphs $x_2=\pm(\epsilon+g(x_1-a))$ with vertical tangents at $x_1=a$ and $x_1=a+1$. The thick segments indicate the flat portions of the boundary.}
\end{figure}
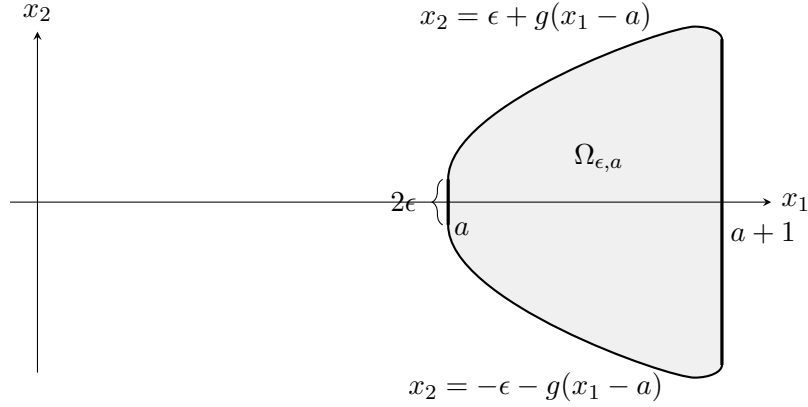

Let $u$ be the (unique) minimizer of $L[\,\cdot\,;\Omega_{\epsilon,a}]$ over $\mathcal K(\Omega_{\epsilon,a})$.
By the symmetry of $\Omega_{\epsilon,a}$ with respect to the $x_1$-axis and uniqueness, $u$ is even in $x_2$.

For $x\in\Omega_{\epsilon,a}$, recall the notation
\[
[x]=\{z\in\overline{\Omega_{\epsilon,a}}:\ u(z)=u(x)+Du(x)\cdot(z-x)\}.
\]
To simplify the notation, we denote by $\Omega_0,\Omega_1,\Omega_2$ the usual stratification of $\Omega_{\epsilon,a}$, and by $\Omega_1'$ the subset of one-ended leaves.
By Remark~\ref{rem:goodC11} applied to $\Omega=\Omega_{\epsilon,a}$,
we have the universal estimate
\begin{equation}\label{c11good}
\|D^2u\|_{L^\infty(\Omega'_{1}\cup\Omega_{2})}\le C_0,
\end{equation}
with a constant $C_0$ independent of $\epsilon$ and $a$.

We will repeatedly use the following distortion bounds at boundary points in the positive phase.

\smallskip
\noindent$\bullet$ If $z\in\partial\Omega_{\epsilon,a}$ satisfies $u(z)>0$ and
$[z]\cap\partial\Omega_{\epsilon,a}=\{z\}$, then there exists a universal constant $C_1>0$ such that
\begin{equation}\label{distortion-good}
\Bigl|\bigl(Du(z)-z\bigr)\cdot \nu_{\Omega_{\epsilon,a}}(z)\Bigr|
\le C_1\,\mathscr H^1([z])
\le C_1\,\diam(\Omega_{\epsilon,a})
\le 5 C_1,
\end{equation}
where we used \eqref{c11good}, \cite[Proposition~5.4]{MRZ20252}, and \eqref{diam-bound}.

\smallskip
\noindent$\bullet$ If $z\in\partial\Omega_{\epsilon,a}$ satisfies $u(z)>0$ and $[z]$
contains a nontrivial segment of $\partial\Omega_{\epsilon,a}$ through $z$, then
\cite[Proposition~2.3]{MRZ20252} implies that
\[
(Du(x)-x)\cdot\nu_{\Omega_{\epsilon,a}}(x)\ge0
\]
along that boundary segment, while \cite[Corollary~A.9]{MRZ20252} (taking $v=\pm1$) yields
\[
\int_{[z]\cap\partial\Omega_{\epsilon,a}}
(Du-x)\cdot\nu_{\Omega_{\epsilon,a}}\,d\mathscr H^1=0.
\]
Hence, in this case,
\begin{equation}\label{distortion-zero}
\bigl(Du(z)-z\bigr)\cdot \nu_{\Omega_{\epsilon,a}}(z)=0.
\end{equation}

In the next subsections we collect a series of results on the structure of minimizers in the
domains $\Omega_{\epsilon,a}$ for $a$ large and $\epsilon$ small. This will then be used to show
that, for suitable parameters, the minimizer of $L[\,\cdot\,;\Omega_{\epsilon,a}]$ is not
$C^{1,1}$ up to the boundary (and actually the modulus of continuity of $Du$ can be as bad as desired).

%------------------------------------------------------------
\subsection{Zeros on the left flat side}

\begin{lem}\label{left0}
Let $\Omega_{\epsilon,a}$ be as above and let $w$ be the minimizer of
$L[\,\cdot\,;\Omega_{\epsilon,a}]$. Set $z_0=(a,0)$. Then $w(z_0)=0$, provided that
$a\ge A$ for some large universal constant $A$.
\end{lem}

\begin{proof}
Since $L[w+t;\Omega_{\epsilon,a}]$ is strictly increasing in $t$ and $w\ge0$, the minimizer must satisfy
$\{w=0\}\neq\emptyset$.
By symmetry of $w$ with respect to the $x_1$-axis, $\{w=0\}$ meets the line $\{x_2=0\}$, so there exists
$b\in[a,a+1]$ such that $w(b,0)=0$.

Assume for contradiction that $w(z_0)>0$. Then necessarily $b>a$ and, by convexity of $t\mapsto w(t,0)$,
\[
\partial_{x_1}w(z_0)=\lim_{h\to0^+}\frac{w(a+h,0)-w(a,0)}{h}<0.
\]
Since $z_0\in\partial\Omega_{\epsilon,a}$ and $\nu_{\Omega_{\epsilon,a}}(z_0)=(-1,0)$, we obtain
\begin{equation}\label{distortionbig}
\bigl(Dw(z_0)-z_0\bigr)\cdot \nu_{\Omega_{\epsilon,a}}(z_0)
=(\partial_{x_1}w(z_0)-a)(-1)
=a-\partial_{x_1}w(z_0)\ge a.
\end{equation}

On the other hand, note that $w(b,0)=0$ with $b>a$ and $\{u=0\}$ is convex and symmetric with respect to the $x_1$-axis.
Then if there exists $y\neq0,\, |y|\le\ez$ for which $w(a,\,y)=0$, then $w(a,\,-y)=0$ as well, and then $w(z_0)=0$, a contradiction. Thus there exists $\delta>0$ such that $w>0$ in the closed slab
$\overline{\Omega_{\epsilon,a}}\cap\{a\le x_1\le a+\delta\}$.
In particular $z_0\notin\Omega_{0}$.
Moreover, for $z$ in this slab the contact set $[z]$ cannot have two distinct endpoints on
$\partial\Omega_{\epsilon,a}$, for otherwise Lemma~\ref{lower-bound-length} would force the triangle generated by $[z]$
and the supporting lines of $\Omega_{\epsilon,a}$ at its endpoints to contain $\{w=0\}$, contradicting $w>0$ in the slab.
Thus $z_0$ satisfies either $[z_0]\cap\partial\Omega_{\epsilon,a}=\{z_0\}$ or $[z_0]$ contains a boundary segment through $z_0$.
In the first case \eqref{distortion-good} gives
\[
\bigl|\bigl(Dw(z_0)-z_0\bigr)\cdot\nu_{\Omega_{\epsilon,a}}(z_0)\bigr|\le 5C_1,
\]
while in the second case \eqref{distortion-zero} yields
\[
\bigl(Dw(z_0)-z_0\bigr)\cdot\nu_{\Omega_{\epsilon,a}}(z_0)=0.
\]
Both contradict \eqref{distortionbig} once $a>5C_1$.
This proves $w(z_0)=0$ for all $a\ge 5C_1+1$.
\end{proof}

\begin{lem}\label{left01}
Let $w$ be the minimizer of $L[\,\cdot\,;\Omega_{\epsilon,a}]$. Set $z_0=(a,0)$, and suppose that
$w(z_0)=0$, $Dw(z_0)=0$, and $w>0$ in the interior of $\Omega_{\epsilon,a}$.
Then
\[
a>\frac{1}{8\epsilon}.
\]
\end{lem}

\begin{proof}
By Armstrong's perturbation argument \cite{A1996}, $\{w=0\}$ cannot be a singleton.
Since $w(z_0)=0$ and $Dw(z_0)=0$, the supporting plane at $z_0$ is identically zero, hence
\[
\{w=0\}\subset [z_0].
\]
By assumption $w>0$ in the interior, so $\{w=0\}\subset \partial\Omega_{\epsilon,a}$.
Since $\{w=0\}$ is convex and contains $z_0$, it must contain a nontrivial segment of the left flat side
$\partial\Omega_{\epsilon,a}\cap\{x_1=a\}$.
Set
\[
\ell:=\mathscr H^1([z_0]\cap\partial\Omega_{\epsilon,a})\in(0,2\epsilon].
\]
For $h>0$ define the truncation
\[
w_h:=(w-h)_+=\max\{w-h,0\},
\qquad S_h:=\{w<h\}.
\]
Then $w_h\in\mathcal K(\Omega_{\epsilon,a})$, $w_h\le w$, $w_h=0$ on $S_h$ and $w_h=w-h$ on $\Omega_{\epsilon,a}\setminus S_h$.
By minimality, $0\le L[w_h;\Omega_{\epsilon,a}]-L[w;\Omega_{\epsilon,a}]$, and a direct computation
(using that $Dw_h=Dw$ a.e.\ on $\Omega_{\epsilon,a}\setminus S_h$ and $Dw_h=0$ a.e.\ on $S_h$)
gives
\[
L[w_h;\Omega_{\epsilon,a}]-L[w;\Omega_{\epsilon,a}]
=\int_{S_h}\Bigl[-\frac12|Dw|^2+x\cdot Dw-w\Bigr]\,dx
-h\,|\Omega_{\epsilon,a}\setminus S_h|.
\]
Since the bracketed term is bounded above by $x\cdot Dw$ and $w$ is $C^1$ up to the boundary, we obtain
\begin{equation}\label{trunc-upper}
0\le L[w_h]-L[w]\le \int_{S_h} x\cdot Dw\,dx - h\,|\Omega_{\epsilon,a}\setminus S_h|.
\end{equation}
Integrating by parts on $S_h$ yields
\[
\int_{S_h} x\cdot Dw\,dx
=\int_{\partial S_h} w\, (x\cdot\nu_{S_h})\,d\mathscr H^1 -2\int_{S_h} w\,dx
\le h\int_{\partial S_h\setminus [z_0]} (x\cdot\nu_{S_h})_+\,d\mathscr H^1,
\]
because $w=0$ on $\partial S_h\cap[z_0]$, and $w\le h$ on $\partial S_h$.

As $h\to0$, the sets $S_h$ converge to $[z_0]$ in Hausdorff distance and in particular
\[
\mathscr H^1(\partial S_h\cap\Omega_{\epsilon,a})\to \ell.
\]
Moreover, $\partial S_h\setminus [z_0]$ converges to $[z_0]$ and its outward normal converges to $e_1$.
Thus, since $[z_0]\subset \{x_1=a\}$,
\[
(x\cdot\nu_{S_h})_+=(x\cdot e_1)_+(1+o(1))=a_+(1+o(1))
\qquad \text{on }\partial S_h\setminus [z_0].
\]
Therefore
\begin{equation}\label{xdubound}
\int_{S_h} x\cdot Dw\,dx \le h\,a_+\,\ell\,(1+o(1))
\qquad\text{as }h\to0.
\end{equation}
On the other hand, since $S_h\to [z_0]$ and $[z_0]$ has zero area, we have
$|\Omega_{\epsilon,a}\setminus S_h|\to|\Omega_{\epsilon,a}|$ as $h\to0$; hence for $h$ small,
\[
|\Omega_{\epsilon,a}\setminus S_h|\ge \frac12|\Omega_{\epsilon,a}|\ge \frac38.
\]
Combining this with \eqref{trunc-upper}--\eqref{xdubound} and dividing by $h$ gives
\[
0\le a_+\,\ell\,(1+o(1))-\frac38.
\]
This implies that $a$ is positive, with $a\ge \frac{3}{8\ell}$.
Since $\ell\le 2\epsilon$, it follows that $a> \frac{1}{8\epsilon}$, as desired.
\end{proof}

\begin{lem}\label{left1}
Let $C_1$ be as in \eqref{distortion-good}, let $a\ge 30C_1$, let $0\le \sigma<a/100$, and let $w$ be the minimizer of $L[\,\cdot\,;\Omega_{\epsilon,a}]$.
Set $z_0=(a,0)$, and suppose that
\[
w(z_0)=0,
\qquad
Dw(z_0)=\sigma e_1.
\]
Then
\[
w(z)=0 \quad \text{and} \quad Dw(z)=\sigma e_1
\qquad \text{for all }z\in \partial\Omega_{\epsilon,a}\cap\{x_1=a\}.
\]
\end{lem}

\begin{proof}
It suffices to prove
\begin{equation}\label{left-side-in-leaf}
\partial\Omega_{\epsilon,a}\cap\{x_1=a\}\subset [z_0].
\end{equation}
Once \eqref{left-side-in-leaf} holds, the identities
\[
w(z)=w(z_0)+Dw(z_0)\cdot(z-z_0)=0
\qquad\text{and}\qquad
Dw(z)=Dw(z_0)=\sigma e_1
\]
follow from the definition of $[z_0]$ and the $C^1$-regularity of $w$.

\smallskip
\noindent\emph{Step 1: two-ended leaves are vertical.}
Suppose $[x]=[x',x'']$ is a segment whose two endpoints lie on $\partial\Omega_{\epsilon,a}$.
Since $w$ is $C^1$, we have $Dw(x')=Dw(x'')$.
If $x'$ lies on the left vertical side $\{x_1=a\}$ and $x''$ lies on the right vertical side $\{x_1=a+1\}$, then
$\nu_{\Omega_{\epsilon,a}}(x')=(-1,0)$ and $\nu_{\Omega_{\epsilon,a}}(x'')=(1,0)$, and
\cite[Proposition~2.3]{MRZ20252} yields
\begin{align*}
&(Dw(x')-x')\cdot(-1,0)\ge0 \ \Longrightarrow\ Dw_1(x')\le a,\\
&(Dw(x'')-x'')\cdot(1,0)\ge0 \ \Longrightarrow\ Dw_1(x'')\ge a+1,
\end{align*}
a contradiction.
Thus a two-ended leaf cannot connect the two vertical sides.

Moreover, note that $\Omega_1^0$ is symmetric with respect to $x_1$ and  that distinct leaves do not cross (cf.\ \cite{MRZ20252}), then these leaves are either parallel to $e_2$-direction, or they do not cross the $x_1$-axis. 
However, by Lemma~\ref{lower-bound-length}, the triangle associated with $[x]$ (as described there) must contain $z_0$. Therefore, the only possibility is that two-ended leaves connect the top and bottom graphs at the same value of $x_1$.

By symmetry of $\Omega_{\epsilon,a}$ about $\{x_2=0\}$, such a leaf must be parallel to the $x_2$-axis.
In particular, for any $x\in \Omega_{1}\setminus\Omega'_{1}$ one has
\begin{equation}\label{vertical-leaf}
[x]=\{x_1\}\cap \Omega_{\epsilon,a}.
\end{equation}

\smallskip
\noindent\emph{Step 2: the left side must lie in $[z_0]$.}
Assume for contradiction that \eqref{left-side-in-leaf} fails, so there exists
$z\in\partial\Omega_{\epsilon,a}\cap\{x_1=a\}$ with $z\notin[z_0]$.
Choose a sequence
\[
z_k\in(\partial\Omega_{\epsilon,a}\cap\{x_1=a\})\setminus[z_0]
\]
converging to an endpoint $\bar z$ of the segment $[z_0]\cap\{x_1=a\}$.
In particular, $\bar z\in[z_0]$.
For each $k$, either $[z_k]\cap\partial\Omega_{\epsilon,a}=\{z_k\}$ or $[z_k]$ contains a boundary segment through $z_k$.
In the first case \eqref{distortion-good} applies, while in the second case \eqref{distortion-zero} applies.
Thus in both cases,
\[
\Bigl|\bigl(Dw(z_k)-z_k\bigr)\cdot \nu_{\Omega_{\epsilon,a}}(z_k)\Bigr|\le 10C_1.
\]
Passing to the limit and using $w\in C^1(\overline{\Omega_{\epsilon,a}})$ gives
\begin{equation}\label{distortionup}
\Bigl|\bigl(Dw(\bar z)-\bar z\bigr)\cdot \nu_{\Omega_{\epsilon,a}}(\bar z)\Bigr|\le 10C_1.
\end{equation}
However, since $\bar z\in[z_0]$ and $\bar z_1=a$, we have $Dw(\bar z)=Dw(z_0)=\sigma e_1$ and
$\nu_{\Omega_{\epsilon,a}}(\bar z)=(-1,0)$, hence
\[
\Bigl|\bigl(Dw(\bar z)-\bar z\bigr)\cdot \nu_{\Omega_{\epsilon,a}}(\bar z)\Bigr|
=
(\sigma-a)(-1)
=a-\sigma
>\frac{99}{100}a
>10C_1,
\]
contradicting \eqref{distortionup} since $a\ge 30C_1$ and $\sigma<a/100$.
This proves \eqref{left-side-in-leaf}.
\end{proof}

%------------------------------------------------------------
\subsection{A one-dimensional region near the left flat side}

\begin{lem}\label{stratification}
Let $C_1$ be as in \eqref{distortion-good}, let $a\ge 30C_1$, let $0\le \sigma<a/100$, and let $w$ be the minimizer of $L[\,\cdot\,;\Omega_{\epsilon,a}]$.
Set $z_0=(a,0)$, and suppose that
\[
w(z_0)=0,
\qquad
Dw(z_0)=\sigma e_1,
\]
and
\begin{equation}\label{not-completely-zero}
\{w=0\}\cap\partial\Omega_{\epsilon,a}
\subset \{x_1\le a+1-\eta_\epsilon\}.
\end{equation}
Then there exists $\delta_\sigma>0$ (possibly depending on $\epsilon$) such that
\begin{equation}\label{piece1d}
x\in \overline{\Omega_{0}}\cup \overline{\Omega_{1}\setminus \Omega'_{1}}
\quad \text{whenever } x=(x_1,x_2)\in \partial  {\Omega_{\epsilon,a}}
\ \text{and}\ a< x_1\le a+\delta_\sigma.
\end{equation}
In particular, if $\sigma=0$, then there exists $0\le \delta<\delta_\sigma$ such that
\begin{equation}\label{pieceflat}
w(x)=0
\quad \text{whenever } x\in \overline{\Omega_{\epsilon,a}}\ \text{and}\ a\le x_1\le a+\delta,
\end{equation}
\begin{equation}\label{piecelinear}
x\in \Omega_{1}\setminus \Omega'_{1}
\quad \text{whenever } x\in \overline{\Omega_{\epsilon,a}}\ \text{and}\ a+\delta< x_1\le a+\delta_\sigma.
\end{equation}
If $\sigma>0$, then $\Omega_{0}=\emptyset$.
\end{lem}

\begin{proof}
By Lemma~\ref{left1},
\begin{equation}\label{z0-line}
\partial\Omega_{\epsilon,a}\cap\{x_1=a\}\subset [z_0].
\end{equation}
In particular, for every $z\in\partial\Omega_{\epsilon,a}\cap\{x_1=a\}$ we have $w(z)=0$ and $Dw(z)=\sigma e_1$, hence
\begin{equation}\label{left-distortion}
\bigl(Dw(z)-z\bigr)\cdot \nu_{\Omega_{\epsilon,a}}(z)=a-\sigma
\qquad\text{for all }z\in\partial\Omega_{\epsilon,a}\cap\{x_1=a\},
\end{equation}
since $\nu_{\Omega_{\epsilon,a}}=(-1,0)$ along the left flat side.

\smallskip
\noindent\emph{Step 1: separation from the good region.}
Assume for contradiction that there exists a sequence of boundary points
\[
z_k\in \partial\Omega_{\epsilon,a}\setminus\Bigl(\overline{\Omega_{0}}\cup\overline{\Omega_{1}\setminus\Omega'_{1}}\Bigr)
\quad\text{with}\quad \dist(z_k,\partial\Omega_{\epsilon,a}\cap\{x_1=a\})\to0.
\]
By \eqref{z0-line}, we may assume $z_k\to \bar z\in \partial\Omega_{\epsilon,a}\cap\{x_1=a\}$.
Since $z_k$ lies outside $\overline{\Omega_{0}}\cup\overline{\Omega_{1}\setminus\Omega'_{1}}$, after discarding finitely many terms we may assume that $w(z_k)>0$ and that either $z_k$ lies in the good region or $[z_k]$ contains a boundary segment through $z_k$.
Hence either \eqref{distortion-good} or \eqref{distortion-zero} applies, and therefore
\[
\Bigl|\bigl(Dw(z_k)-z_k\bigr)\cdot \nu_{\Omega_{\epsilon,a}}(z_k)\Bigr|\le 10C_1.
\]
Letting $k\to\infty$ and using $w\in C^1(\overline{\Omega_{\epsilon,a}})$ yields
\[
\Bigl|\bigl(Dw(\bar z)-\bar z\bigr)\cdot \nu_{\Omega_{\epsilon,a}}(\bar z)\Bigr|\le 10C_1,
\]
which contradicts \eqref{left-distortion} because $a\ge 30C_1$ and $\sigma<a/100$.
Therefore,
\[
\dist\!\left(\partial\Omega_{\epsilon,a}\cap\{x_1=a\},\,
\partial\Omega_{\epsilon,a}\setminus\Bigl(\overline{\Omega_{0}}\cup\overline{\Omega_{1}\setminus\Omega'_{1})}\Bigr)\right)>0.
\]
Choosing $\delta_\sigma>0$ smaller than this distance proves \eqref{piece1d}.

\smallskip
\noindent\emph{Step 2: the structure when $\sigma=0$.}
When $\sigma=0$, it follows by Lemma~\ref{left1} that $w=0$ and $Dw=0$ on the entire left flat side.

We claim that $\Omega_{1}\setminus\Omega'_{1}\neq\emptyset$.
Indeed, suppose by contradiction that $\Omega_{1}\setminus\Omega'_{1}=\emptyset$.
Then \eqref{not-completely-zero} implies that $
\partial\Omega_{\epsilon,a}\setminus \overline{\Omega_{0}}\neq\emptyset$ and 
we can choose a sequence
\[
\partial\Omega_{\epsilon,a}\setminus\Bigl(\overline{\Omega_{0}}\cup\overline{\Omega_{1}\setminus\Omega'_{1}}\Bigr)
\ni
z_k \to z \in \overline{\Omega_0}\cap \partial\Omega_{\epsilon,a}
\subset \{x_1\le a+1-\eta_\epsilon\}.
\]
Therefore,
\[
\bigl(Dw(z)-z\bigr)\cdot \nu_{\Omega_{\epsilon,a}}(z)
=
-\,z\cdot \nu_{\Omega_{\epsilon,a}}(z)
\ge \frac{a}{\sqrt2} \ge 15C_1,
\]
which contradicts the bound
\[
\bigl|\bigl(Dw(z)-z\bigr)\cdot \nu_{\Omega_{\epsilon,a}}(z)\bigr|\le 10C_1
\]
(cp. Step 1). This proves the claim.

Define
\[
a_1:=\inf\{x_1:\ (x_1,0)\in \Omega_{1}\setminus \Omega'_{1}\},
\qquad \delta:=a_1-a\ge 0.
\]
By \eqref{not-completely-zero} we must have $a_1\le a+1-\eta_\epsilon$ (otherwise we again produce a sequence as in Step~1).
By definition of $a_1$ and \eqref{piece1d}, every point with $a\le x_1<a_1$ belongs to $\Omega_{0}$, hence $w=0$ there,
which gives \eqref{pieceflat}. Similarly, for $a_1<x_1\le a+\delta_\sigma$ the points must lie in
$\Omega_{1}\setminus\Omega'_{1}$, giving \eqref{piecelinear}.

\smallskip
\noindent\emph{Step 3: the case $\sigma>0$.}
If $\sigma>0$, then by Lemma~\ref{left1} we have $w=0$ and $Dw=\sigma e_1$ on the left flat side.
For any $z$ on this side and any $t>0$ small with $z+t e_1\in\Omega_{\epsilon,a}$, convexity gives
\[
w(z+t e_1)\ge w(z)+Dw(z)\cdot(t e_1)=\sigma t>0.
\]
Thus $w>0$ in a full strip adjacent to $\{x_1=a\}$.
Since $\{w=0\}$ is convex and contains the left flat side, this rules out any other zero point in $\Omega_{\epsilon,a}$,
hence $\{w=0\}$ has empty interior and therefore $\Omega_{0}=\emptyset$.
\end{proof}

%------------------------------------------------------------
\subsection{One-dimensional reduction and blow-up of $v''$ when $\delta=0$}

From now on, assume that Lemma~\ref{stratification} applies with $\sigma=0$, and denote the corresponding minimizer by $u$
(i.e., $u$ is the $w$ above for this choice of parameters). Let $\delta$ and $\delta_\sigma$ be as in
\eqref{pieceflat}--\eqref{piecelinear}. Then on the slab $a<x_1<a+\delta_\sigma$, every point belongs to
$\Omega_{0}$ or $\Omega_{1}\setminus\Omega'_{1}$. In particular, by \eqref{vertical-leaf}, $u$ is constant along vertical
segments in the $x_2$-direction on $\{a+\delta<x_1<a+\delta_\sigma\}$, hence $u$ is one-dimensional there:
there exists $\delta_0\in(0,\delta_\sigma]$ and a convex function $v$ such that
\begin{equation}\label{u-to-v}
u(x_1,x_2)=v(x_1)\qquad\text{for }(x_1,x_2)\in \Omega_{\epsilon,a}\ \text{with }a<x_1<a+\delta_0,
\end{equation}
where $v(a)=0$, $v'(a)=0$, and
\[
v(x_1)=0 \ \text{for }a\le x_1\le a+\delta,
\qquad
v(x_1)>0 \ \text{for }a+\delta<x_1\le a+\delta_0.
\]
Define
\[
G(x_1):=\int_a^{x_1} g(t-a)\,dt,
\qquad
h(x_1):=\epsilon+g(x_1-a).
\]
Then
\[
G(a)=0,
\qquad
G'(x_1)=g(x_1-a).
\]
Because $u$ is even in $x_2$, it is convenient to restrict to the upper half-domain
\[
\Omega_{\epsilon,a}^+:=\Omega_{\epsilon,a}\cap\{x_2\ge0\}
\]
(this only changes the energy by an irrelevant factor $2$).
Then, on the region $\Omega_{\epsilon,a}^+\cap\{a+\delta<x_1<a+\delta_0\}$, the energy of $u$ reduces to
\begin{equation}\label{reduced-energy}
\int_{a+\delta}^{a+\delta_0}
\left[
\frac{1}{2} h(x_1)\,(v'(x_1)-x_1)^2
+\frac{1}{6} h(x_1)^3
+v(x_1)\,h(x_1)
\right]\,dx_1,
\end{equation}
and the corresponding Euler--Lagrange equation is
\begin{equation}\label{EL-v}
-\bigl[(v'(x_1)-x_1)h(x_1)\bigr]'+h(x_1)=0
\qquad\text{for }a+\delta<x_1<a+\delta_0.
\end{equation}
Integrating once, we obtain
\begin{equation}\label{vprime}
(v'(x_1)-x_1)\,h(x_1)=G(x_1)+\epsilon x_1 + C_\delta
\end{equation}
for some constant $C_\delta\in\mathbb R$, hence
\begin{equation}\label{vsecond-general}
v''(x_1)
=2-\frac{g'(x_1-a)\bigl[G(x_1)+\epsilon x_1+C_\delta\bigr]}{\bigl[g(x_1-a)+\epsilon\bigr]^2}.
\end{equation}
Since $v\equiv 0$ on $[a,a+\delta]$ and $v$ is $C^1$, we have the matching condition
\[
v'(a+\delta)=0.
\]
Using \eqref{vprime} at $x_1=a+\delta$ gives
\begin{equation}\label{Cdelta}
C_\delta
=-(a+\delta)\bigl(2\epsilon+g(\delta)\bigr)-G(a+\delta).
\end{equation}
Substituting \eqref{Cdelta} into \eqref{vsecond-general} yields
\begin{equation}\label{vsecond}
v''(x_1)
=
2+\frac{-G(x_1)+G(a+\delta)-\epsilon x_1+(a+\delta)\bigl(2\epsilon+g(\delta)\bigr)}{\bigl[g(x_1-a)+\epsilon\bigr]^2}\,
g'(x_1-a).
\end{equation}
In the special case $\delta=0$, we have $g(\delta)=g(0)=0$ and $G(a+\delta)=G(a)=0$, hence
\begin{equation}\label{c11blowup}
v''(x_1)
=
2+\frac{-G(x_1)+\epsilon(2a-x_1)}{\bigl[g(x_1-a)+\epsilon\bigr]^2}\,g'(x_1-a)
\to +\infty
\qquad\text{as }x_1\to a^+,
\end{equation}
because $g'(x_1-a)\to+\infty$ as $x_1\to a^+$ and the numerator stays positive for $x_1$ sufficiently close to $a$.
Therefore, if for some choice of parameters one has $\delta=0$ in \eqref{pieceflat}, then necessarily
\[
\|D^2u\|_{L^\infty(\Omega_{\epsilon,a})}=+\infty.
\]
Actually, when $\delta=0$, one can say more. By \eqref{vprime} and \eqref{Cdelta},
\[
v'(x_1)
=
x_1+\frac{\int_0^{x_1-a} g(t)\,dt+\epsilon x_1-2a\epsilon}{g(x_1-a)+\epsilon}
=
\frac{g(x_1-a)\,x_1+\int_0^{x_1-a} g(t)\,dt+2\epsilon(x_1-a)}{g(x_1-a)+\epsilon}.
\]
In particular, if  we define
\[
N(y):=(a+y)g(y)+\int_0^y g(t)\,dt+2\epsilon y,
\qquad 0\le y<\delta_0,
\]
then
\[
v'(a+y)=\frac{N(y)}{g(y)+\epsilon},
\qquad
v'(a)=0.
\]
Since $g(0)=0$, it follows that
\begin{equation}
\label{eq:v'g}
|v'(a+y)-v'(a)|=\frac{|N(y)|}{g(y)+\epsilon} =\frac{a}{\epsilon}g(y)+O_{a,\epsilon}(y).
\end{equation}
Thus, for any modulus of continuity $\omega$, if 
$$\frac{g(r)}{\omega(r)}\to +\infty\quad \text{ as} \quad r\to 0,$$ 
then $v'$ also fails to have the same modulus of continuity.

%------------------------------------------------------------
\subsection{Positivity of the slope at the left boundary for large $a$}

We start with the following auxiliary lemma. 

\begin{lem}\label{no 1d zero}
Let $w$ be a minimizer of $L[\cdot;\,\Omega_{\ez,\,a}]$ and set $z_0=(a,\,0)$. Then  $\{w=0\}$ cannot be a subset of the $x_1$-axis. 
\end{lem}
\begin{proof}
We prove the lemma by a contradiction argument. 

Suppose that
$$[z_0]\subset\{x_2=0\}.$$
We exclude this by a truncation argument \`a la Armstrong (cp. \cite{A1996}). For \(\delta>0\), define
\[
S_\delta:=\{ w<\delta\},
\qquad
v_\delta:=\max\{w,\delta\},
\qquad
\bar w_\delta:=v_\delta-\delta.
\]
Note that
$$
L[\bar w_\delta;\Omega_{\epsilon,a}]
=
L[v_\delta;\Omega_{\epsilon,a}]
-\delta\,|\Omega_{\epsilon,a}|.
$$
Then, 
by the minimality of \(w\),
applying the variational identity \eqref{var-id} with the affine function \(p\equiv\delta\), we get
\begin{align*}
0&\leq L[v_\delta;\Omega_{\epsilon,a}]-L[w;\Omega_{\epsilon,a}]=L[v_\delta;\Omega_{\epsilon,a}]-L[w;\Omega_{\epsilon,a}]-\delta\,|\Omega_{\epsilon,a}|\\
&=
-\frac12\int_{S_\delta}|Dw|^2\,dx
+\int_{\partial S_\delta\cap\partial\Omega_{\epsilon,a}}
(\delta-w)\,(-x\cdot\nu_{\Omega_{\epsilon,a}})\,d\mathscr H^1 
+3\int_{S_\delta}(\delta-w)\,dx-\delta\,|\Omega_{\epsilon,a}|,
\end{align*}
and therefore
\begin{equation}\label{eq:min delta}
    0\le
C\,\delta\,
\mathscr H^1(\partial S_\delta\cap\partial\Omega_{\epsilon,a})+3\delta\,|S_\delta|
-\delta\,|\Omega_{\epsilon,a}|.
\end{equation}
Since \(\{w=0\}\cap\Omega_{\epsilon,a}\) is a segment contained inside ${x_2=0}$,
it follows that
$$\partial S_\delta\cap\partial\Omega_{\epsilon,a}
\subset
\{x\in\partial\Omega_{\epsilon,a}:\ w(x)\le\delta\}\subset \partial\Omega_{\epsilon,a}\cap \{ |x_2|\leq o_\delta(1)\},
$$
hence
\[
|S_\delta|\to0\quad \text{and}\quad \mathscr H^1(\partial S_\delta\cap\partial\Omega_{\epsilon,a})\to0
\qquad\text{as }\delta\to 0.
\]
Combined with  \eqref{eq:min delta} this proves that
$$
0\le
o(\delta)-\delta\,|\Omega_{\epsilon,a}|,
$$
a contradiction for $\delta \ll 1$.
Thus the claim of the lemma follows. 
\end{proof}

\begin{lem}\label{pslope}
Let $w$ be the minimizer of $L[\,\cdot\,;\Omega_{\epsilon,a}]$, and set $z_0=(a,0)$.
For any fixed $\epsilon\in(0,1)$, we have
\[
Dw(z_0)\cdot e_1>0,
\]
provided that $a$ is sufficiently large.
\end{lem}

\begin{proof}
Assume by contradiction that there exists a sequence $a_k\to\infty$ such that, letting $w_k$ denote the minimizer of
$L[\,\cdot\,;\Omega_{\epsilon,a_k}]$ and $z_{0,k}:=(a_k,0)$, we have
\[
Dw_k(z_{0,k})\cdot e_1\le0.
\]
By Lemma~\ref{left0}, for $k$ large we have $w_k(z_{0,k})=0$.
Since $w_k\ge0$ in $\Omega_{\epsilon,a_k}$ and $w_k$ is $C^1$, the one-sided derivative in the inward direction $e_1$ satisfies
\[
Dw_k(z_{0,k})\cdot e_1
=\lim_{h\to0^+}\frac{w_k(a_k+h,0)-w_k(a_k,0)}{h}
=\lim_{h\to0^+}\frac{w_k(a_k+h,0)}{h}\ge0.
\]
Hence $Dw_k(z_{0,k})\cdot e_1=0$, and by symmetry $Dw_k(z_{0,k})=0$.

Fix such a large $k$ and write $a:=a_k$, $w:=w_k$, and $z_0:=z_{0,k}$.
We now show that this leads to a competitor with strictly lower energy for $a$ large.
There are two cases.

\smallskip
\noindent\emph{Case 1:} we first consider the case when
\begin{equation}\label{zero-left}
\{w=0\}\cap\partial\Omega_{\epsilon,a}\subset\{x_1\le a+1-\eta_\epsilon\}.
\end{equation}
Then Lemma~\ref{stratification} applies with $\sigma=0$, and $w$ has the one-dimensional form \eqref{u-to-v}
(on a maximal slab $\{a<x_1<a+\delta_0\}$ with $\delta_0>0$).
If $\delta_0<1-\eta_\epsilon$, then since $w$ is a function depends only on $x_1$ according to Lemma~\ref{stratification}, 
the explicit formula \eqref{vprime}--\eqref{Cdelta} gives
\[
(x-Dw)\cdot e_1 \ge c\,\epsilon a
\]
at the terminal point $\{x\in \Omega\colon x_1=a+\delta_0\}$ of the one-dimensional region, while the universal distortion estimate
\eqref{distortion-good} bounds the same quantity by $10C_1$.
For $a$ large this is impossible.
Consequently, for $a$ sufficiently large,
\begin{equation}\label{u-1d-left}
w\ \text{is independent of $x_2$ in}\ \ \Omega_{\epsilon,a}\cap\{x_1\le a+1-\eta_\epsilon\}.
\end{equation}

\smallskip
\noindent\emph{Case 2:} we now consider the case when \eqref{zero-left} fails.
Then there exists
\begin{equation}\label{assump large zero}
x'\in\{w=0\}\cap\partial\Omega_{\epsilon,a}\cap\{x_1>a+1-\eta_\epsilon\}.
\end{equation}
By symmetry
\[
x''=(x'_1,-x'_2)\in\{w=0\},
\]
and by convexity of $\{w=0\}$ we obtain that $w$ vanishes on $[x',x'']$.
Using again Lemma~\ref{lower-bound-length} and the foliation property of the leaves in $\Omega_{1}$ (see \cite{MRZ20252}), one concludes that no two-ended leaf can cross this segment. In particular, this implies
\begin{equation}\label{w 0}
    w=0 \quad \text{throughout } \Omega_{\epsilon,a}\cap\{x_1\le a+1-\eta_\epsilon\}.
\end{equation}
Indeed, suppose by contradiction that
\[
(\Omega_2\cup\Omega'_1)\cap \Omega_{\epsilon,a}\cap\{x_1\le a+1-\eta_\epsilon\}\neq \emptyset.
\]
Then there exists
\[
z\in \partial\Omega_{\epsilon,a}\cap \overline{\Omega_2\cup\Omega'_1}\cap \{w=0\}\cap\{x_1\le a+1-\eta_\epsilon\}.
\]
Thus, since $\{w=0\}$ is a $2$-dimensional convex set (by Lemma~\ref{no 1d zero} and the assumption \eqref{assump large zero}) and $w\in C^1(\overline{\Omega})$, one has $Dw(z)=0$.
Also, applying \eqref{distortion-good} and \eqref{distortion-zero} to a sequence $z_k\in \partial\Omega_{\epsilon,a}\cap \overline{\Omega_2\cup\Omega'_1}\backslash\{w=0\}$ converging to $z$, we have 
\[
|(Dw(z)-z)\cdot \nu_{\Omega_{\epsilon,a}}(z)|\le 10 C_1.
\]
This yields a contradiction when $a\ge 30C_1$, and proves \eqref{w 0}.
In particular, \eqref{u-1d-left} holds trivially.

\smallskip
In either case, \eqref{u-1d-left} holds.
We next record a global gradient bound.
On $\{x_1\le a+1-\eta_\epsilon\}$, $w$ is one-dimensional, so
\[
Dw(x)=v'(x_1)e_1.
\]
Moreover, on this region the boundary condition
\[
(Dw-x)\cdot\nu_{\Omega_{\epsilon,a}}\ge0
\qquad\text{on }\partial\Omega_{\epsilon,a}
\]
implies $v'(x_1)\le x_1$.
In particular,
\[
Dw\cdot e_1=v'(x_1)\le x_1\le a+1
\qquad\text{on }\Omega_{\epsilon,a}\cap\{x_1\le a+1-\eta_\epsilon\}.
\]

\smallskip
\noindent\emph{Hessian bound on the right strip $\{x_1\ge a+1-\eta_\epsilon\}$.}
Let $x=(x_1,x_2)\in \Omega_{\epsilon,a}$ with $x_1\ge a+1-\eta_\epsilon$.

\smallskip
\noindent\emph{Case A: $x\in\Omega_{0}$.}
Then $w$ is affine in a neighborhood of $x$, hence $D^2w(x)=0$.

\smallskip
\noindent\emph{Case B: $x\in \Omega'_{1}\cup\Omega_{2}$.}
By \eqref{c11good} we have $\|D^2w(x)\|\le C_0$.

\smallskip
\noindent\emph{Case C: $x\in (\Omega_{1}\setminus\Omega'_{1})\cap\{x_1\ge a+1-\eta_\epsilon\}$.}
Then $[x]$ is a segment with two endpoints on $\partial\Omega_{\epsilon,a}$.
By \eqref{vertical-leaf}, this leaf is vertical:
\[
[x]=\{x_1\}\cap\Omega_{\epsilon,a},
\]
and its endpoints are
\[
x^\pm=(x_1,\pm g_{\epsilon,a}(x_1))\in\partial\Omega_{\epsilon,a}.
\]
Set
\[
s:=x_1-a\in[1-\eta_\epsilon,1),
\qquad
\tau:=e_2.
\]
We distinguish two subcases.

\smallskip
\noindent\emph{Subcase C1: $g'(s)<0$.}
Let $P^\pm$ be the tangent lines to the upper and lower graphs at $x^\pm$.
Since
\[
P^+:\ y_2=g_{\epsilon,a}(x_1)+g'(s)(y_1-x_1),
\qquad
P^-:\ y_2=-g_{\epsilon,a}(x_1)-g'(s)(y_1-x_1),
\]
their intersection has first coordinate
\[
y_1=x_1-\frac{g_{\epsilon,a}(x_1)}{g'(s)}>x_1.
\]
Hence the triangular region enclosed by $[x]$, $P^+$, and $P^-$ is contained in
$\{y_1\ge x_1\}$.
Since $z_0=(a,0)\in\{w=0\}$ and $x_1>a$, this contradicts
Lemma~\ref{lower-bound-length}. Therefore this subcase cannot occur.

\smallskip
\noindent\emph{Subcase C2: $g'(s)\ge0$.}
Since $g$ is concave and $g'(t)=1$ on $(\frac12,1-\eta_\epsilon)$, we have
\[
0\le g'(s)\le 1.
\]
Hence the outward unit normals $\nu^\pm$ to $\partial\Omega_{\epsilon,a}$ at $x^\pm$ satisfy
\[
|\tau\cdot \nu^\pm|
=\frac{1}{\sqrt{1+g'(s)^2}}
\ge \frac{1}{\sqrt2}.
\]
Moreover,
\[
|x^+-x^-|=2g_{\epsilon,a}(x_1)
\ge 2g_{\epsilon,a}(a+1-\eta_\epsilon)
\ge c(\epsilon)>0.
\]
% Thus one endpoint, say $x^+$, is a good endpoint with uniform length and angle constants.

Let
\[
\ell_x(y):=w(x)+Dw(x)\cdot(y-x),
\qquad
h_r:=\sup_{B_r(x)}(w-\ell_x),
\]
% If the untilted section already meets $\partial\Omega_{\epsilon,a}$ only near $x^+$,
% we work with the untilted construction; otherwise 
and we tilt with respect to the 
endpoint $x^-$ for which
$$|x-x^-|\le |x-x^+|$$
as in Lemma~\ref{lem:local-perturb}. Then for all sufficiently
small $r$ we obtain a section $S_r$ such that
\[
\frac{h_r}{r^2}
\le
C\Bigl(\|Dw\|_{L^\infty(\Omega_{\epsilon,a})}
+\sup_{\Omega_{\epsilon,a}}|y|\Bigr)
\left(
1+\frac{\mathscr H^1(\partial S_r\cap\partial\Omega_{\epsilon,a})}{|S_r|}
\right).
\]
Furthermore, by \cite[Proof of Theorem~4.1]{MRZ20252}, for $r$ small the boundary trace
\[
A_r:=\partial S_r\cap\partial\Omega_{\epsilon,a}
\]
is contained in a fixed boundary chart around $x^+$, and  
$x$ satisfies
\[
|x-x^+|\ge c(\epsilon).
\]
Let $P^+$ be the tangent line to the upper graph at $x^+$.
Since $|\tau\cdot \nu^+|\ge 1/\sqrt2$, we get
\[
\dist(x,P^+)\ge c(\epsilon).
\]
If $A_r\neq\emptyset$, then by convexity
\[
\operatorname{co}(A_r\cup\{x\})\subset \overline S_r.
\]
Also, because $A_r$ lies in a boundary chart with uniformly bounded slope, its length is
comparable to the length of its orthogonal projection onto $P^+$.
Therefore
\[
|S_r|
\ge \bigl|\operatorname{co}(A_r\cup\{x\})\bigr|
\ge c(\epsilon)\,\mathscr H^1(A_r).
\]
Hence
\[
\mathscr H^1(\partial S_r\cap\partial\Omega_{\epsilon,a})
\le C(\epsilon)\,|S_r|,
\]
and therefore
\begin{equation}\label{tau-r2-est}
h_r
\le
C\bigl(\|Dw\|_{L^\infty(\Omega_{\epsilon,a})}+a+3\bigr)\,r^2.
\end{equation}
This implies that
\[
\|D^2w(x)\|
\le C\bigl(\|Dw\|_{L^\infty(\Omega_{\epsilon,a})}+a+3\bigr)
\]
at every point of second differentiability.

\smallskip
\noindent\emph{Case D: $x\in \partial\{w=0\}\cap\Omega_{\epsilon,a}\cap\{x_1\ge a+1-\eta_\epsilon\}$.}
In this case, since $w(x)=0$, then $x\notin \Omega_{1}$. In particular,  $x$ belongs to the closure of a connected
component $U$ of
\[
\Omega_{\epsilon,a}\setminus(\Omega_{1}\setminus\Omega'_1).
\]
Applying Proposition~\ref{prop:goodC11}, we get
\[
\|D^2w\|_{L^\infty(\overline U\cap\Omega_{\epsilon,a})}\le C_0,
\]
therefore $D^2w$ is bounded at $x$.

\smallskip
Combining Cases A--D, we conclude that
\[
\|D^2w\|_{L^\infty(\Omega_{\epsilon,a}\cap\{x_1\ge a+1-\eta_\epsilon\})}
\le C_0\bigl(\|Dw\|_{L^\infty(\Omega_{\epsilon,a})}+a+3\bigr).
\]
Now, using this bound and connecting any point in $\{x_1>a+1-\eta_\epsilon\}$ to a point in $\{x_1\le a+1-\eta_\epsilon\}$ by a segment of length
$\le 2\eta_\epsilon$, we deduce (for $a$ large)
\[
\|Dw\|_{L^\infty(\Omega_{\epsilon,a})}
\le C_0\bigl(\|Dw\|_{L^\infty(\Omega_{\epsilon,a})}+a+3\bigr)\eta_\epsilon + (a+1).
\]
Since $\eta_\epsilon=\frac{1}{1000C_0}\epsilon\le \frac{1}{1000C_0}$, this yields
\begin{equation}\label{Du-bound}
\|Dw\|_{L^\infty(\Omega_{\epsilon,a})}\le 3a
\qquad\text{for $a$ sufficiently large}.
\end{equation}
Finally, define the admissible competitor
\[
\hat w(x):=w(x)+\varphi(x),
\qquad \varphi(x):=x_1-a.
\]
Since $x_1\ge a$ on $\Omega_{\epsilon,a}$, we have $\hat w\ge w\ge0$ and $\hat w$ is convex, hence $\hat w\in\mathcal K(\Omega_{\epsilon,a})$.
A direct computation gives
\begin{align*}
L[w;\Omega_{\epsilon,a}]-L[\hat w;\Omega_{\epsilon,a}]
&=
\int_{\Omega_{\epsilon,a}}
\left(
- D\varphi \cdot (Dw-x)
- \frac{|D\varphi|^2}{2}
- \varphi
\right)\,dx\\
&=
\int_{\Omega_{\epsilon,a}}
\Bigl((x-Dw)\cdot e_1 -\frac12-(x_1-a)\Bigr)\,dx.
\end{align*}
On $\{x_1\le a+1-\eta_\epsilon\}$, the one-dimensional structure \eqref{u-1d-left} and the explicit formula
\eqref{vprime}--\eqref{Cdelta} imply
\[
(x-Dw)\cdot e_1\ge c\,\epsilon a
\]
for $a$ large, while on the complement strip
$\{x_1>a+1-\eta_\epsilon\}$ we have the crude bound
\[
\bigl|(x-Dw)\cdot e_1\bigr|\le 5a
\]
by \eqref{Du-bound}.
Since the strip $\{x_1>a+1-\eta_\epsilon\}$ has area $O(\eta_\epsilon)$ and $|\Omega_{\epsilon,a}|$ is bounded below,
we conclude that
\[
L[w;\Omega_{\epsilon,a}]-L[\hat w;\Omega_{\epsilon,a}]>0
\]
for all sufficiently large $a$, contradicting the minimality of $w$.
Therefore $Dw(z_0)\cdot e_1>0$ for $a$ large.
\end{proof}

%------------------------------------------------------------
\subsection{Existence of a blow-up configuration}

\begin{thm}\label{thm:counterexample}
Let $u$ be the minimizer of $L[\,\cdot\,;\Omega_{\epsilon,a}]$. For any sufficiently small $\epsilon$, there exists $a>0$ such that
$u(z_0)=0$ and $Du(z_0)=0$, and there exists $\delta_0>0$ such that
\[
u(x_1,x_2)=v(x_1)>0
\]
for some one-dimensional convex function $v$ whenever $(x_1,x_2)\in\Omega_{\epsilon,a}$ and $a<x_1<a+\delta_0$.
Moreover, given a modulus of continuity $\omega$, if
\begin{equation}
\label{eq:g omega}\frac{g(r)}{\omega(r)}\to \infty\quad \text{ as} \quad r\to 0,
\end{equation}
then $Du$ also fails to have the same modulus of continuity.
\end{thm}

\begin{proof}
Fix $\epsilon<\frac{1}{240C_1}$.
By Lemma~\ref{pslope}, choose $a_0\ge 30C_1$ sufficiently large such that the minimizer $u_{a_0}$ of
$L[\,\cdot\,;\Omega_{\epsilon,a_0}]$ satisfies
\[
Du_{a_0}(a_0,0)\cdot e_1=\alpha_0>0.
\]
Since $\Omega_{\epsilon,a}$ is obtained from $\Omega_{\epsilon,0}$ by translation in the $e_1$-direction, and the functional
depends on $(Dw-x)$ in a translation-covariant way (cf.\ \cite{MRZ20252}), one checks that if $w$ is a minimizer on
$\Omega_{\epsilon,a}$, then for each $\beta>0$ the function
\[
x\mapsto w(x+\beta e_1)-\beta x_1
\]
is a minimizer on $\Omega_{\epsilon,a-\beta}$ up to addition of a constant (which does not affect the gradient).

Set
\[
a:=a_0-\alpha_0,
\qquad
\tilde u(x):=u_{a_0}(x+\alpha_0 e_1)-\alpha_0(x_1-a).
\]
Then \(\tilde u\) is a minimizer of \(L[\cdot\,;\Omega_{\epsilon,a}]\), and at
$z_0:=(a,0)$
we have
\[
\tilde u(z_0)=u_{a_0}(a_0,0)=0,
\qquad
D\tilde u(z_0)=Du_{a_0}(a_0,0)-\alpha_0 e_1=0.
\]
Moreover, since
\[
(Du_{a_0}(a_0,0)-a_0e_1)\cdot(-e_1)\ge0,
\]
we have \(\alpha_0\le a_0\), hence \(a\ge0\).

We claim that \(\tilde u>0\) in the interior of \(\Omega_{\epsilon,a}\).
Indeed, since \(u_{a_0}\) is strictly positive in the interior in a neighborhood to the right of the left boundary, \cite[Theorem~1(i)]{MRZ20252} implies that the contact set of \(u_{a_0}\) with its tangent plane at \((a_0,0)\) has zero Lebesgue measure. By translation covariance, the same is true for the contact set \([z_0]\) of \(\tilde u\). 
Now suppose by contradiction that
\[
\{\tilde u=0\}\cap \Omega_{\epsilon,a}\neq\emptyset.
\]
Since \(\tilde u(z_0)=0\) and \(D\tilde u(z_0)=0\), it follows that
$\{\tilde u=0\}\subset [z_0].$
Observe now that, because \(\tilde u\) is even in \(x_2\), if \(y=(y_1,y_2)\in \{\tilde u=0\}\cap\Omega_{\epsilon,a}\) with \(y_2\neq0\), then also
$y^*:=(y_1,-y_2)\in \{\tilde u=0\}.$
By convexity of \([z_0]\), the triangle \(\operatorname{co}(\{z_0,y,y^*\})\) would then be contained in \([z_0]\), contradicting \(|[z_0]|=0\).
This proves that 
\[
\{\tilde u=0\}\cap \Omega_{\epsilon,a}
\]
is a segment contained in the \(x_1\)-axis with endpoint \(z_0\).
However, this is impossible by Lemma~\ref{no 1d zero}.

Thus \(\tilde u>0\) in the interior of \(\Omega_{\epsilon,a}\), so we can apply Lemma~\ref{left01} to \(\tilde u\) to obtain
\[
a=a_0-\alpha_0\ge \frac{1}{8\epsilon}\ge 30C_1.
\]
Replacing \(\tilde u\) by \(u\) from now on, we may therefore assume that
\[
u(z_0)=0,\qquad Du(z_0)=0,\qquad u>0 \ \text{in the interior of } \Omega_{\epsilon,a},
\qquad a\ge 30C_1.
\]
Applying Lemma~\ref{stratification} to $u$ with $\sigma=0$ and using interior positivity near $z_0$ forces the flat
region length $\delta$ in \eqref{pieceflat} to be $\delta=0$ (otherwise $u$ would vanish on interior points).
Thus $u$ is one-dimensional and positive immediately to the right of $a$:
there exists $\delta_0>0$ such that
\[
u(x_1,x_2)=v(x_1)>0\qquad\text{for }a<x_1<a+\delta_0,
\]
for some one-dimensional convex function $v$.
Since $\delta=0$, formula \eqref{c11blowup} yields
\[
v''(x_1)\to+\infty
\qquad\text{as }x_1\to a^+.
\]
In addition, if \eqref{eq:g omega} holds,
then \eqref{eq:v'g} implies that also $Du$ fails to have the same modulus of continuity. 
\end{proof}

\section{Free boundary between $\Omega_1'$ and $\Omega_2$}\label{sec4}

In this section we assume that $\partial \Omega$ is of class $C^{1,1}$ and that $u \in C^{1,1}(\overline\Omega)$.

Recall that $\Omega_1'$ is the union of leaves with one endpoint on $\partial\Omega$
and the other in the interior of $\Omega$. The tamed free boundary $\mathcal{T}$ was
defined in \eqref{tamefree}. Fix $z_0\in\mathcal{T}$ and set
\[
x_0:=[z_0]\cap\partial\Omega.
\]
Parametrize $\partial\Omega$ locally by a $C^{1,1}$ map
\[
\gamma\colon (-\varepsilon,\varepsilon)\to \partial\Omega,
\qquad
\gamma(0)=x_0.
\]
By \cite[Lemma~6.1 and Remark~6.4]{MRZ20252}, the leaves $[\gamma(t)]$ spread out as they
leave the boundary and foliate a neighborhood of $x_0$. In particular, every point $x$
sufficiently close to $z_0$ can be written uniquely as
\[
x=\gamma(t)+r\,\xi(t),
\]
for a unique pair $(t,r)$, where $\xi(t)$ is the unit direction of the leaf $[\gamma(t)]$ and
$R(t)>0$ is its length, so that the leaf terminates at $\gamma(t)+R(t)\xi(t)$.

Moreover, the following Euler--Lagrange equations hold near $\gamma(0)=x_0$
(see \cite{MRZ20252}):
\begin{equation}\label{EL r}
R^2(t)\,|\dot\xi(t)|
=
2\,|\dot\gamma(t)|\,
\bigl(Du(\gamma(t))-\gamma(t)\bigr)\cdot \nu_{\Omega}(\gamma(t)),
\end{equation}
\begin{equation}\label{EL u}
3-\Delta u
=
\frac{3r-2R(t)}
{r+|\dot\xi(t)|^{-1}\,\xi(t)\times\dot\gamma(t)}
=:f(x).
\end{equation}
Here $a\times b$ denotes the scalar cross product in $\mathbb R^2$.
The raywise change of variables
\[
(x_1,x_2)\mapsto (r,t)
\]
is locally bi-Lipschitz by \cite[Lemma~6.3 and Corollary~6.6]{MRZ20252}.

Let $v$ be the minimal convex extension of $u$ restricted to a connected component of
$B_{\varepsilon_0}(x_0)\cap\Omega$, and set
$$
U:=u-v.
$$
Then $U\ge0$ and
\begin{equation}\label{w obstacle}
\Delta U=f(x)\,\chi_{\{U>0\}}\ge c_0\,\chi_{\{U>0\}}
\end{equation}
for some $c_0>0$. Moreover, for some sufficiently small $\rho>0$,
\[
\{U>0\}\cap B_\rho(z_0)=\Omega_2\cap B_\rho(z_0),
\qquad
\{U=0\}\cap B_\rho(z_0)=\Omega_1\cap B_\rho(z_0).
\]
Since $R$ is continuous by \cite[Lemma~7.4]{MRZ20252}, it follows from \eqref{EL r} that
$|\dot\xi|$ is continuous, and then $f$ is continuous by \eqref{EL u}. Hence $U$ solves a
classical obstacle problem with continuous right-hand side, uniformly bounded away from $0$.

Therefore, by Caffarelli's alternative
(see \cite{C1977,C1998} and \cite[Theorem~0.3]{B2001}), for every $x\in\mathcal T$ one has
\[
\text{either }\ \lim_{r\to0}\frac{|B_r(x)\cap\Omega_1|}{|B_r(x)|}=0,
\qquad\text{or}\qquad
\lim_{r\to0}\frac{|B_r(x)\cap\Omega_1|}{|B_r(x)|}=\frac12.
\]

\subsection{Regularity at $\frac12$-density points}

Let $x\in\mathcal T$ be a point with density $\frac12$.
Then the free boundary is a Reifenberg vanishing set in a neighborhood of $x$.
More precisely, there exists $\rho>0$ such that for every compact set
$K\subset B_\rho(x)$,
\[
\lim_{r\to0}\theta_K(r)=0,
\qquad
\theta_K(r)
:=
\sup_{0<s\le r}\,
\sup_{y\in \mathcal T\cap K}
\frac{d_H\bigl(P_{y,s}\cap B_s(y),\,\mathcal T\cap B_s(y)\bigr)}{s},
\]
where $P_{y,s}$ is a line through $y$ and $d_H$ denotes the Hausdorff distance.
Moreover, the set of $\frac12$-density points is relatively open in
$\mathcal T\cap B_\rho(z_0)$.

By Reifenberg's theorem \cite{R1960}
(see also \cite[Section~3]{DKT2001}),
$\mathcal T$ can be locally parametrized by a bi-H\"older homeomorphism.
Since the raywise change of variables is bi-Lipschitz
\cite[Corollary~6.6]{MRZ20252}, the graph representation of $\mathcal T$ in
$(t,r)$-coordinates implies that $R$ is locally H\"older.
In addition, since leaves in $\Omega_1$ do not intersect, the map
\[
t\longmapsto \gamma(t)+R(t)\,\xi(t)
\]
is locally injective.

We now prove uniqueness of blow-ups at a $\frac12$-density point.

\begin{lem}\label{goodpoint}
Let $x\in\mathcal T$ be a point with density $\frac12$. Then the blow-up limit
\[
P_x:=\lim_{r\to0}\frac{\mathcal T-x}{r}
\]
exists and is unique. Moreover, if $P_x$ is not tangent to the leaf $[x]$ at $x$,
then $\mathcal T$ is smooth in a neighborhood of $x$.
\end{lem}

\begin{proof}
Let $I:=[x]$ and denote by $\xi(x)$ the unit direction of $I$.
For any sequence $r_k\downarrow0$, after passing to a subsequence we have
\[
\mathcal T_k:=\frac{\mathcal T-x}{r_k}\to P
\]
locally in Hausdorff distance, where $P$ is a line because $x$ is a $\frac12$-density point.
Let $\nu_P$ be the unit normal to $P$, chosen so that $\nu_P\cdot\xi(x)\ge0$.

If for every such sequence one has $\nu_P\cdot\xi(x)=0$, then $P$ must be the line containing
$I$, and the blow-up is unique.

Assume now that there exists a sequence $r_k\downarrow0$ for which the limiting line $P$
satisfies
\begin{equation}\label{nontangent0}
\nu_P\cdot\xi(x)\ge c_\ast>0.
\end{equation}
Then there exists $k_0$ such that
\begin{equation}\label{pslope1}
\nu_{P_{x,r_k}}\cdot\xi(x)\ge \frac12\,c_\ast
\qquad\text{for all }k\ge k_0,
\end{equation}
where $P_{x,r}$ denotes a line through $x$ realizing the Reifenberg flatness at scale $r$.
Set
\[
r_0:=r_{k_0}.
\]

Because of the non-tangency \eqref{pslope1}, the free boundary can be represented as a graph
of the leaf-length function $R$ over the leaf parameter $t$ in $B_{r_0}(x)$, with Lipschitz
constant depending only on $c_\ast$.
In particular,
\[
|R(y)-R(x)|\le C(c_\ast)\,r_0
\qquad\text{for all }y\in \mathcal T\cap B_{r_0}(x).
\]
Using \eqref{EL r} and \eqref{EL u}, this implies
\begin{equation}\label{iteration0}
|f(y)-f(x)|\le C_1\,r_0
\qquad\text{for all }y\in B_{r_0}(x),
\end{equation}
for some constant $C_1$ depending only on the local geometry of $\Omega$,
$\|\gamma\|_{C^{1,1}}$, $\|u\|_{C^{1,1}}$, and $c_\ast$.

Since $\mathcal T$ is Reifenberg vanishing near $x$, given $\delta_0>0$ we may choose $r_0$
small enough so that
\begin{equation}\label{inismall}
\theta_{\overline B_\rho}(r_0)\le \delta_0.
\end{equation}
Let
\[
P_0:=P_{x,r_0}.
\]
Then
\begin{equation}\label{iteration1}
d_H\bigl(P_0\cap B_{r_0}(x),\,\mathcal T\cap B_{r_0}(x)\bigr)
\le \theta_{\overline B_\rho}(r_0)\,r_0.
\end{equation}

Set $\rho_0:=1/2$.
By \eqref{iteration0} and \eqref{inismall}, we can apply \cite[Theorem~7.1]{B2001} to obtain,
for $r_0$ sufficiently small, a line $P_1$ through $x$ such that
\begin{equation}\label{iteration2}
d_H\bigl(P_1\cap B_{\rho_0 r_0}(x),\,\mathcal T\cap B_{\rho_0 r_0}(x)\bigr)
\le C\,C_1\,r_0^2.
\end{equation}
Combining \eqref{iteration1} and \eqref{iteration2}, and rescaling to unit size, yields
\[
d_H\bigl(P_1\cap B_1(x),\,P_0\cap B_1(x)\bigr)
\le \frac{2C\,C_1r_0+2\theta_{\overline B_\rho}(r_0)}{\rho_0}.
\]

Iterating this argument at scales $\rho_0^m r_0$, we obtain a sequence of lines $P_m$ through $x$
such that
\[
d_H\bigl(P_{m+1}\cap B_{\rho_0^{m+1}r_0}(x),\,\mathcal T\cap B_{\rho_0^{m+1}r_0}(x)\bigr)
\le C\,C_1\,\rho_0^{2m}\,r_0^2,
\]
and
\[
d_H\bigl(P_{m+1}\cap B_1(x),\,P_m\cap B_1(x)\bigr)
\le \frac{2C\,C_1\,\rho_0^{2m-2}\,r_0^2}{\rho_0}.
\]
Hence $P_m$ converges to a limit line $P_\infty$.

Moreover,
\[
d_H\bigl(P_m\cap B_1(x),\,P_\infty\cap B_1(x)\bigr)
\le
\sum_{l=m}^\infty
\frac{2C\,C_1\,\rho_0^{2l-2}\,r_0^2}{\rho_0}
\le
\frac{2C\,C_1\,\rho_0^{2m-3}\,r_0^2}{1-\rho_0^2},
\]
and
\begin{equation}\label{iterationf}
d_H\bigl(P_0\cap B_1(x),\,P_\infty\cap B_1(x)\bigr)
\le
\frac{2C\,C_1r_0+2\theta_{\overline B_\rho}(r_0)}{\rho_0}
+\frac{2C\,C_1\,\rho_0^{-1}\,r_0^2}{1-\rho_0^2}.
\end{equation}
Therefore, given $\varepsilon>0$, choosing $r_0$ sufficiently small gives
\begin{equation}\label{itfinal}
d_H\bigl(P_0\cap B_1(x),\,P_\infty\cap B_1(x)\bigr)\le \varepsilon.
\end{equation}
Taking $\varepsilon\ll c_\ast$ and using \eqref{pslope1}, we infer that
\begin{equation}\label{nontangent}
\nu_{P_\infty}\cdot\xi(x)\ge \frac14\,c_\ast>0.
\end{equation}
Thus, under \eqref{nontangent0}, the blow-up is unique and equals $P_\infty$.

Finally, once \eqref{nontangent} holds, the graph function $R$ is locally Lipschitz, hence in
particular locally H\"older. By \eqref{EL r} and \eqref{EL u}, the coefficient $f$ is then locally
H\"older continuous near $x$. Standard elliptic regularity yields
\[
u\in C^{2,\alpha}(\Omega_2\cap B_\eta(x))
\]
for some $\alpha\in(0,1)$ and $\eta>0$, and the regularity theory for the obstacle problem
implies that $\mathcal T$ is a local $C^{1,\alpha}$ graph near $x$
(see \cite{C1977,C1998,K1976,KN1977,B2001}).
Bootstrapping as in \cite[Proposition~7.6]{MRZ20252} gives local smoothness of $\mathcal T$.
\end{proof}

\begin{lem}\label{T-C1}
The tamed free boundary $\mathcal T$ is $C^1$ at every $\frac12$-density point.
\end{lem}

\begin{proof}
Fix $x_0\in\mathcal T$ with density $\frac12$.
If
\[
\nu_{P_{x_0}}\cdot\xi(x_0)>0,
\]
then Lemma~\ref{goodpoint} already shows that $\mathcal T$ is smooth in a neighborhood of $x_0$.

It remains to treat the case
\[
\nu_{P_{x_0}}\cdot\xi(x_0)=0.
\]
We claim that whenever $x_k\in\mathcal T$ and $x_k\to x_0$, one has
\[
P_{x_k}\to P_{x_0}.
\]
This gives continuity of the tangent line field, hence $C^1$ regularity at $x_0$.

Suppose the claim fails. Passing to a subsequence, there exists $c_\ast>0$ such that
\begin{equation}\label{gap0}
\nu_{P_{x_k}}\cdot\xi(x_k)\ge 2c_\ast
\qquad\text{for all }k.
\end{equation}

\medskip
\noindent\emph{Claim.}
There exists $r_\ast>0$ such that
\begin{equation}\label{claim1}
\nu_{P_{x_k,r}}\cdot\xi(x_k)\ge c_\ast
\qquad\text{for all }0<r\le r_\ast\ \text{and all }k.
\end{equation}

\noindent
Indeed, if \eqref{claim1} failed, we could find $r_m\downarrow0$ and indices $k_m$ such that
\[
\nu_{P_{x_{k_m},r_m}}\cdot\xi(x_{k_m})<c_\ast.
\]
Since $P_{x_{k_m},r}\to P_{x_{k_m}}$ as $r\to0$ and \eqref{gap0} holds, there exists
$\bar r_m\in(0,r_m)$ such that
\begin{equation}\label{gap1}
c_\ast\le \nu_{P_{x_{k_m},\bar r_m}}\cdot\xi(x_{k_m})<\frac54\,c_\ast.
\end{equation}
Now apply \eqref{itfinal} from the proof of Lemma~\ref{goodpoint} at the point $x_{k_m}$ and
scale $\bar r_m$, with $\varepsilon\ll c_\ast$.
For $m$ large, this gives
\[
d_H\bigl(P_{x_{k_m},\bar r_m}\cap B_1(x_{k_m}),\,P_{x_{k_m}}\cap B_1(x_{k_m})\bigr)\le \varepsilon.
\]
Combining this with \eqref{gap1} yields
\[
\nu_{P_{x_{k_m}}}\cdot\xi(x_{k_m})<\frac32\,c_\ast,
\]
contradicting \eqref{gap0}. This proves the claim.

\medskip
Fix $r\in(0,r_\ast)$. By compactness and Reifenberg flatness, after passing to a subsequence
we may assume
\[
P_{x_k,r}\to Q_r,
\]
where $Q_r$ is a line through $x_0$ approximating $\mathcal T$ in $B_r(x_0)$.
Since $\xi$ varies continuously along the foliation,
passing to the limit in \eqref{claim1} gives
\[
\nu_{Q_r}\cdot\xi(x_0)\ge c_\ast.
\]
Letting $r\to0$ and using the uniqueness of the blow-up at $x_0$
(from Lemma~\ref{goodpoint}) yields
\[
\nu_{P_{x_0}}\cdot\xi(x_0)\ge c_\ast,
\]
contradicting the assumption $\nu_{P_{x_0}}\cdot\xi(x_0)=0$.
Therefore $P_{x_k}\to P_{x_0}$, and $\mathcal T$ is $C^1$ at $x_0$.
\end{proof}

\subsection{Rigidity at $0$-density points}

We now turn to cusp-type blow-ups. The next argument proves the rigidity theorem
stated in the introduction.

\begin{proof}[Proof of Theorem~\ref{Rigi}]
Write
$x=(x_1,x_2)=(r\cos\theta,r\sin\theta).$
Since $u\ge0$ in $B_1$ and $u=0$ on
$S=\{(x_1,0): -1<x_1<0\},$
every point of $S$ is a local minimum. Hence
\begin{equation}\label{eq:grad-on-S}
D u(x_1,0)=0
\qquad\text{for all }-1<x_1\le0.
\end{equation}
Because $u\in C^{1,1}(B_1)$, the gradient is globally Lipschitz: there exists $L\ge1$ such that
\begin{equation}\label{eq:Lip}
|D u(x)-D u(y)|\le L|x-y|
\qquad\forall x,y\in B_1.
\end{equation}
Convexity implies $D^2u\ge0,$
while the identity $\Delta u=1$ in $D$ implies that each first derivative $\partial_i u$
is harmonic in $D$.

\medskip
\noindent\textit{Step 1: a one-sided lower bound for $\partial_1u$.}
Let
\[
y:=\Bigl(-\frac12,0\Bigr)\in S.
\]
Since $u$ is convex, the gradient is monotone:
\[
(D u(x)-D u(y))\cdot (x-y)\ge0.
\]
Using \eqref{eq:grad-on-S}, this gives
\[
\partial_1u(x)\Bigl(x_1+\frac12\Bigr)+\partial_2u(x)\,x_2\ge0.
\]
Thus, whenever $|x_1|\le1/4$,
\[
\partial_1u(x)\ge -4x_2\,\partial_2u(x).
\]
Now let $x\in \partial D\cap B_{1/4}$. Since the sides of $D$ lie in $\{x_1<0\}$, the point
$(x_1,0)$ belongs to $S$, and \eqref{eq:Lip} together with \eqref{eq:grad-on-S} gives
\[
|\partial_2u(x)|
\le |D u(x)-D u(x_1,0)|
\le L|x_2|.
\]
Hence
\[
\partial_1u(x)\ge -4Lx_2^2
\qquad\text{for all }x\in \partial D\cap B_{1/4}.
\]
Set
$C_\ast:=4L.$
Since $u_{11}\ge0$, for each fixed $x_2$ the function
\[
x_1\longmapsto \partial_1u(x_1,x_2)+C_\ast x_2^2
\]
is nondecreasing. Therefore, choosing $r_0=r_0(a)\in(0,1/4]$ so that every horizontal segment
through a point of $D\cap B_{r_0}$ meets $\partial D\cap B_{1/4}$, the previous boundary estimate
propagates into the sector:
\begin{equation}\label{eq:basic-lower}
\partial_1u(x)\ge -C_\ast x_2^2
\qquad\text{for all }x\in D\cap B_{r_0}.
\end{equation}

\medskip
\noindent\textit{Step 2: a harmonic lower barrier for $u_{11}$.}
Fix $a'\in(a,\pi/2)$ and set
\[
D':=\{(r,\theta): -\pi+a'<\theta<\pi-a'\}\subset D,
\qquad
b:=\frac{\pi}{2(\pi-a')}\in\Bigl(\frac12,1\Bigr),
\]
\[
\hat w(r,\theta):=r^b\cos(b\theta).
\]
Then $\hat w$ is harmonic in $D'$, vanishes on $\partial D'$, and is strictly positive in $D'$.
Moreover,
\begin{equation}\label{eq:d1w}
\partial_1\hat w
=
\cos\theta\,\partial_r\hat w-\frac{\sin\theta}{r}\,\partial_\theta\hat w
=
b\,r^{b-1}\cos\bigl((b-1)\theta\bigr),
\qquad
|\partial_1\hat w|\le b\,r^{b-1}.
\end{equation}
Let
$\hat v:=u_{11}.$
Then $\hat v$ is harmonic and nonnegative in $D$.
If $\hat v\equiv0$ in $D$, then $\partial_1u$ is harmonic and independent of $x_1$, hence
\[
\partial_1u(x)=\alpha x_2+\beta
\]
in $D$ for some constants $\alpha,\beta$.
Since $\partial_1u=0$ on $S$, we have $\beta=0$.
The convexity matrix
\[
\begin{pmatrix}
0 & \alpha\\
\alpha & u_{22}
\end{pmatrix}
\]
must be nonnegative semidefinite, so necessarily $\alpha=0$.
Hence $\partial_1u\equiv0$ in $D$.
Since now $\Delta u=u_{22}=1$ in $D$, and $u(0)=|D u(0)|=0$, we conclude that
\[
u(x)=\frac{x_2^2}{2}
\qquad\text{in }D.
\]
So it only remains to exclude the case $\hat v\not\equiv0$.

Assume therefore that $\hat v\not\equiv0$.
By the strong maximum principle, $\hat v>0$ in $D$.
In particular,
\[
m:=\min_{\theta\in[-\pi+a',\,\pi-a']}\hat v(1/2,\theta)>0.
\]
Set
\[
c_0:=m\,2^b,
\qquad
h:=\hat v-c_0\hat w
\quad\text{in }D'\cap B_{1/2}.
\]
Then $h$ is harmonic. On the two sides of $D'$ one has $\hat w=0$, so $h\ge0$ there, while
on the circular arc $\partial B_{1/2}\cap \overline{D'}$,
\[
c_0\hat w\le m\,2^b \Bigl(\frac12\Bigr)^b=m\le \hat v.
\]
Hence $h\ge0$ on $\partial(D'\cap B_{1/2})$, and therefore
\begin{equation}\label{eq:v-lower}
u_{11}(x)=\hat v(x)\ge c_0\hat w(x)
\qquad\text{for all }x\in D'\cap B_{1/2}.
\end{equation}

\medskip
\noindent\textit{Step 3: a nonlinear subharmonic barrier.}
Choose $\gamma\in(1,1/b)$ and define
\[
\delta:=1-b\gamma\in(0,1).
\]
For $\varepsilon>0$, set
\[
\psi_\varepsilon(x):=\varepsilon\,\hat w(x)^\gamma-(C_\ast+1)x_2^2.
\]
Since $\hat w$ is harmonic and positive in $D'$,
\[
\Delta(\hat w^\gamma)
=
\gamma(\gamma-1)\hat w^{\gamma-2}|D \hat w|^2
=
\gamma(\gamma-1)b^2\,r^{b\gamma-2}\cos^{\gamma-2}(b\theta)
\ge \gamma(\gamma-1)b^2\,r^{-(1+\delta)},
\]
because $b\gamma-2=-(1+\delta)$ and $\cos(b\theta)\in(0,1]$ in $D'$.
Since $\Delta(x_2^2)=2$, it follows that
\[
\Delta\psi_\varepsilon
\ge
\varepsilon\,\gamma(\gamma-1)b^2\,r^{-(1+\delta)}-2(C_\ast+1).
\]
Now choose
\begin{equation}\label{eq:r0eps}
\varepsilon:=A\,r_0^{1+\delta},
\qquad
A:=\frac{2(C_\ast+1)}{\gamma(\gamma-1)b^2}.
\end{equation}
Then $\psi_\varepsilon$ is subharmonic in $D'\cap B_{r_0}$.

Set
\[
U:=D'\cap\bigl(\{|x_1|<C_0r_0\}\times \{|x_2|<r_0\}\bigr),
\qquad
C_0:=\cot a'+1,
\]
and define
\[
H:=\partial_1u-\psi_\varepsilon.
\]

\medskip
\noindent\textit{Step 4: $H\ge0$ on $\partial U$.}

\smallskip
\noindent\emph{(a) On the sides of $D'$.}
There $\hat w=0$, so by \eqref{eq:basic-lower},
\[
H=\partial_1u+(C_\ast+1)x_2^2
\ge -C_\ast x_2^2+(C_\ast+1)x_2^2
\ge0.
\]

\smallskip
\noindent\emph{(b) On the top and bottom edges.}
We treat the top edge $\{x_2=r_0\}$, the bottom one being identical.
Set
\[
t_\ell:=-r_0\cot a',
\qquad
t_r:=(\cot a'+1)r_0,
\]
and define
\[
F(t):=H(t,r_0)
\qquad\text{for }t\in[t_\ell,t_r].
\]
Using \eqref{eq:v-lower},
\[
F'(t)
=
u_{11}(t,r_0)-\partial_1\psi_\varepsilon(t,r_0)
\ge c_0\hat w(t,r_0)-\varepsilon\,\partial_t\bigl(\hat w(t,r_0)^\gamma\bigr).
\]
There exist constants $0<c_1\le c_2$, depending only on $a'$, such that
for $t=t_\ell+s$ one has
\begin{equation}\label{eq:lin-growth}
c_1\,r_0^{b-1}s\le \hat w(t,r_0)\le c_2\,r_0^{b-1}s,
\qquad 0\le s\le t_r-t_\ell.
\end{equation}
Also, since $\hat w(t_\ell,r_0)=0$, \eqref{eq:basic-lower} gives
\[
F(t_\ell)=\partial_1u(t_\ell,r_0)+(C_\ast+1)r_0^2\ge r_0^2.
\]
Integrating from $t_\ell$ to $t=t_\ell+s$, and using \eqref{eq:r0eps} and \eqref{eq:lin-growth},
we obtain
\begin{align}
F(t)
&\ge r_0^2 + c_0\int_{t_\ell}^{t}\hat w(\tau,r_0)\,d\tau
-\varepsilon\,\hat w(t,r_0)^\gamma \nonumber\\
&\ge r_0^2+\frac{c_0c_1}{2}\,r_0^{b-1}s^2
-Ac_2^\gamma\,r_0^{1+\delta+\gamma(b-1)}\,s^\gamma \nonumber\\
&= r_0^2+\frac{c_0c_1}{2}\,r_0^{b-1}s^2
-Ac_2^\gamma\,r_0^{2-\gamma}\,s^\gamma. \label{eq:layer-ineq}
\end{align}
Choose $\theta_0>0$ small, depending only on $A,c_2,\gamma$, so that
$F\ge0$ on $[t_\ell,t_\ell+\theta_0r_0]$.
On the complementary interval, the positive term in \eqref{eq:layer-ineq}
is of order $r_0^{1+b}$ while the negative term is of order $r_0^2$;
since $1+b<2$, the positive contribution dominates for $r_0$ sufficiently small.
Hence $F\ge0$ on the whole top edge.

\smallskip
\noindent\emph{(c) On the right side.}
Fix $x_2\in[-r_0,r_0]$ and let
\[
t_\ell(x_2):=-|x_2|\cot a'
\]
be the left endpoint of the horizontal section of $D'$ at height $x_2$.
Starting from the boundary value $H(t_\ell(x_2),x_2)\ge0$ given by part (a), and repeating
the argument from part (b) along the horizontal segment from $t_\ell(x_2)$ to $t_r$, we obtain
\[
H(t_r,x_2)\ge0
\]
for all $|x_2|\le r_0$.
Therefore $H\ge0$ on $\partial U$.

\medskip
\noindent\textit{Step 5: interior comparison and contradiction.}
Inside $U$, the function $\partial_1u$ is harmonic while $\psi_\varepsilon$ is subharmonic,
so $H=\partial_1u-\psi_\varepsilon$ is superharmonic.
Thus, by the minimum principle and Step~4 we deduce that
$H\ge0$ inside $U$, and therefore
\[
\partial_1u\ge \psi_\varepsilon
\qquad\text{in }U.
\]
In particular, for $x=(x_1,0)$ with $x_1>0$ sufficiently small,
\[
\partial_1u(x_1,0)\ge \varepsilon\,\hat w(x_1,0)^\gamma
=
\varepsilon\,x_1^{b\gamma}.
\]
Since $b\gamma<1$, this contradicts the Lipschitz continuity of $\partial_1u$ at the origin
together with $\partial_1u(0)=0$.
Therefore the alternative $\hat v\not\equiv0$ is impossible, and we are in the rigid case
\[
u(x)=\frac{x_2^2}{2}
\qquad\text{in }D\cap B_1.
\]
In particular, $u$ cannot be strictly convex in $D\cap B_1$.
\end{proof}

\begin{proof}[Proof of Theorem~\ref{Fbregu}]
For a minimizer $u$, recall that $v$ denotes the minimal convex extension of $u$ restricted to a connected component of $B_{\varepsilon_0}(x_0)\cap\Omega$, and that $U:=u-v$ satisfies the obstacle problem \eqref{w obstacle}.

By applying Caffarelli's alternative to $U$, every point of $\mathcal T$ has density either $0$ or $\frac12$. Moreover, Lemma~\ref{T-C1} shows that $\mathcal T$ is $C^1$ at every $\frac12$-density point.

It remains to exclude $0$-density points. Suppose by contradiction that $0\in \mathcal T$ is a $0$-density point. Up to a rotation, we may assume that the leaf $I$ containing $0$ is parallel to the $x_1$-axis and that $u|_I=\ell$, where $\ell$ is an affine function.

Fix $\varepsilon>0$. Since the blow-up of $U$ at a $0$-density point has constant Laplacian and is of the form $\frac{c_0}{2}x_2^2$ for some $c_0>0$, the uniform convergence of the blow-up sequence implies that there exists $r_\varepsilon>0$ such that, for every $0<r<r_\varepsilon$,
\[
\Omega'_1\cap B_r \subset B_r \cap \{x\in \Omega : -\varepsilon r \le x_1 \le \varepsilon r,\ x_2 \le 0\}.
\]

Therefore, given $a>0$, by choosing $\varepsilon=\varepsilon(a)>0$ sufficiently small and $0<r<r_\varepsilon$, the function
\[
\frac{(u-\ell)(rx)}{3r^2}
\]
satisfies the assumptions of Theorem~\ref{Rigi}, since $\Delta u=3$ in $\Omega_2$. This contradicts the strict convexity of $u$ in $\Omega_2$, and hence $0$-density points cannot occur on $\mathcal T$.

Consequently, every point of $\mathcal T$ is a $\frac12$-density point, and $\mathcal T$ is locally a $C^1$ embedded curve.
\end{proof}


\begin{thebibliography}{99}




\bibitem{A1996}
M. Armstrong. \emph{Multiproduct nonlinear pricing}. Econometrica 64 (1996), No. 1 , 51--75.

% \bibitem{B1945}
% R. H. Bing, \emph{Generalizations of two theorems of Janiszewski}. Bull. Amer. Math. Soc. 51, (1945). 954--960.

\bibitem{B2001}
I. Blank, 
\emph{Sharp results for the regularity and stability of the free boundary in the obstacle problem.}
Indiana Univ. Math. J. 50 (2001), no. 3, 1077--1112.



 

\bibitem{C1977}
L. A. Caffarelli. \emph{The regularity of free boundaries in higher dimensions}. Acta Math. 139 (1977), no. 3-4, 155--184.

\bibitem{C1998}
L. A. Caffarelli. \emph{The obstacle problem revisited}. J. Fourier Anal. Appl. 4 (1998), no. 4-5, 383--402.

\bibitem{CL2006}
L.A. Caffarelli and P.-L. Lions. Untitled notes. Unpublished, 2006+.

\bibitem{CL2001} 
G. Carlier, T. Lachand-Robert, \emph{Regularity of solutions for some variational problems subject to a convexity constraint}.  Comm. Pure Appl. Math. 54 (2001), no. 5, 583--594


\bibitem{CL2001a} 
G. Carlier, T. Lachand-Robert, and B. Maury. A numerical approach to variational
problems subject to convexity constraint. Numer. Math., 88(2):299-318, 2001.

\bibitem{DKT2001}
G. David, C.  Kenig, T. Toro, 
\emph{Asymptotically optimally doubling measures and Reifenberg flat sets with vanishing constant}.
Comm. Pure Appl. Math. 54 (2001), no. 4, 385--449.

\bibitem{EM2010} 
Ivar Ekeland and Santiago Moreno-Bromberg. An algorithm for computing solutions
of variational problems with global convexity constraints. Numer. Math., 115(1):45-69,
2010.

\bibitem{FS2019}
A. Figalli, J. Serra, 
\emph{On the fine structure of the free boundary for the classical obstacle problem.}
Invent. Math. 215 (2019), no. 1, 311--366.

\bibitem{K1976}
D. Kinderlehrer, \emph{The free boundary determined by the solution to a differential equation}.
Indiana Univ. Math. J. 25 (1976), no. 2, 195--208.

\bibitem{KN1977}
D. Kinderlehrer,   L.  Nirenberg, \emph{Regularity in free boundary problems}. Ann. Scuola Norm. Sup. Pisa Cl. Sci. (4) 4 (1977), no. 2, 373--391.


\bibitem{KLM2005}
H. Koch, G. Leoni, M. Morini, 
\emph{On optimal regularity of free boundary problems and a conjecture of De Giorgi. }
Comm. Pure Appl. Math. 58 (2005), no. 8, 1051--1076.

\bibitem{MO2014} 
Quentin M\'erigot and \'Edouard Oudet. Handling convexity-like constraints in variational
problems. SIAM J. Numer. Anal., 52(5):2466-2487, 2014.

\bibitem{Mc2026}
R. J. McCann, L. D. O'Brien, C. Rankin, 
\emph{On the Hausdorff dimension and singularities of the monopolist's free boundary curve}. arXiv:2603.14100.

\bibitem{MRZ2025}
R. J. McCann, C. Rankin, K. S. Zhang, 
\emph{$C^{1,1}$-regularity for principal-agent problems}.
Adv. Math. 478 (2025), Paper No. 110396.

\bibitem{MRZ20252}
R. J. McCann, C. Rankin, K. S. Zhang, 
\emph{The monopolist's free boundary problem in the plane}. arXiv:2412.15505.

\bibitem{MZ2024}
R. J. McCann, K. S. Zhang, \emph{A duality and free boundary
approach to adverse selection}. Math. Models Methods Appl. Sci., 34(12):2351--2394, 2024. 

\bibitem{R1960}
E. R. Reifenberg, \emph{Solution of the Plateau Problem for $m$-dimensional surfaces of varying topological type}.
Acta Math. 104 (1960), 1--92.

\bibitem{RC98} 
J.-C. Rochet and P. Chon\'e. Ironing, sweeping and multidimensional screening. Econometrica,
66:783-826, 1998.


% \bibitem{T1995}
% T. Toro, \emph{Geometric conditions and existence of bi-Lipschitz parameterizations}.
% Duke Math. J. 77 (1995), no. 1, 193--227.

\end{thebibliography}
\end{document}